\documentclass[11pt]{amsart}
\usepackage{amscd,amssymb}
\usepackage{amsthm,amsmath,amssymb}
\usepackage[matrix,arrow]{xy}

\theoremstyle{definition}
\newtheorem{example}[equation]{Example}

\newtheorem{theorem}[equation]{Theorem}
\newtheorem{lemma}[equation]{Lemma}
\newtheorem{corollary}[equation]{Corollary}
\newtheorem{conjecture}[equation]{Conjecture}

\newtheorem*{question*}{Question}
\newtheorem{problem}[equation]{Problem}
\newtheorem*{problem*}{Problem}

\theoremstyle{remark}
\newtheorem{remark}[equation]{Remark}

\makeatletter\@addtoreset{equation}{section} \makeatother

\author{Ivan Cheltsov and Dimitra Kosta}

\title{Computing $\alpha$-invariants of singular del Pezzo surfaces}

\address{\begin{tabbing}
\hspace*{28 em}\=\kill School of Mathematics, University of Edinburgh, Edinburgh EH9 3JZ, UK\\
\medskip
\texttt{I.Cheltsov@ed.ac.uk}\\
\medskip
\texttt{dimitrakosta@hotmail.com}
\end{tabbing}}

\thanks{The authors thank G.\,Brown, N.\,Budur, J.\,Koll\'ar,
 M.\,Mustata, J.\,Park, Y.\,Prokhorov for valuable
comments.}

\thanks{The authors would like to thank an anonymous referee for many
useful remarks.}

\thanks{This paper was completed under financial support provided by
IKY (Greek State Scholarship Foundation).}

\begin{document}

\begin{abstract}
We prove new local inequality for divisors on surfaces and utilize
it to compute $\alpha$-invariants of singular del Pezzo surfaces,
which implies that del Pezzo surfaces of degree one whose singular
points are of type $\mathbb{A}_{1}$, $\mathbb{A}_{2}$,
$\mathbb{A}_{3}$, $\mathbb{A}_{4}$, $\mathbb{A}_{5}$ or
$\mathbb{A}_{6}$ are K\"ahler-Einstein.
\end{abstract}

\maketitle

We assume that all varieties are projective, normal, and defined
over $\mathbb{C}$.

\section{Introduction}
\label{section:intro}

Let $X$ be a~Fano variety with at most quotient singularities (a
Fano orbifold).

\begin{theorem}[{\cite{Tian90}}]
\label{theorem:Tian} If $\mathrm{dim}(X)=2$ and $X$ is smooth,
then
$$
\text{the~surface}\ X\ \text{is K\"ahler--Einstein}\ \iff\ \text{the~group}\ \mathrm{Aut}\big(X\big)\ \text{is reductive}.%
$$
\end{theorem}

An important role in the~proof of Theorem~\ref{theorem:Tian} is
played by several holomorphic invariants, which are now known as
$\alpha$-invariants. Let us describe their algebraic counterparts.

Let $D$ be an~effective $\mathbb{Q}$-divisor on the~variety $X$.
Then the~number
$$
\mathrm{c}\big(X,D\big)=\mathrm{sup}\Big\{\epsilon\in\mathbb{Q}\ \Big|\text{the~log pair}\ \big(X, \epsilon D\big)\ \text{is log canonical}\Big\}\in\mathbb{Q}\cup\big\{+\infty\big\}.%
$$
is called the~log canonical threshold of the~divisor $D$ (see
\cite[Definition~8.1]{Ko97}). Put
$$
\mathrm{lct}_{n}\big(X\big)=\mathrm{inf}\Bigg\{\mathrm{c}\Bigg(X,\frac{1}{n}B\Bigg)\ \Bigg|\ B\ \text{is a~divisor in $\big|-nK_{X}\big|$}\Bigg\}%
$$
for every $n\in\mathbb{N}$. For small $n$, the~number
$\mathrm{lct}_{n}(X)$ is usually not very hard to compute.

\begin{example}[{\cite{Pa01}}]
\label{example:del-pezzo-degree-3} If $X$ is a~smooth surface in
$\mathbb{P}^3$ of degree $3$, then
$$
\mathrm{lct}_{1}\big(X\big)=\left\{%
\aligned
&2/3\ \text{if $X$ has an~Eckardt point},\\%
&3/4\ \text{if $X$ has no Eckardt points}.\\%
\endaligned\right.%
$$
\end{example}

The number $\mathrm{lct}_{n}(X)$ is denoted by $\alpha_n(X)$ in
\cite{Ti90b}.

\begin{remark}
\label{remark:semi-continuity} It follows from
\cite[Lemma~4.8]{Mustata02} that the~set
$$
\Bigg\{\mathrm{c}\Bigg(X,\frac{1}{n}B\Bigg)\ \Bigg|\ B\ \text{is a~divisor in $\big|-nK_{X}\big|$}\Bigg\}%
$$
is finite (cf. \cite{Lyubeznik}). Thus, there exists a~divisor
$B\in |-nK_{X}|$ such that
$\mathrm{lct}_{n}(X)=\mathrm{c}(X,B/n)\in\mathbb{Q}$.
\end{remark}

If the~variety $X$ is smooth, then it is proved by Demailly (see
\cite[Theorem~A.3]{ChSh08c}) that
$$
\mathrm{inf}\Big\{\mathrm{lct}_{n}\big(X\big)\ \Big|\ n\in\mathbb{N} \Big\}=\alpha\big(X\big),%
$$
where $\alpha(X)$ is the~$\alpha$-invariant introduced by Tian in
\cite{Ti87}. Put
$\mathrm{lct}(X)=\mathrm{inf}\{\mathrm{lct}_{n}(X)\ |\
n\in\mathbb{N}\}$.

\begin{conjecture}[{\cite[Question~1]{Ti90b}}]
\label{conjecture:stabilization} There is an~$n\in\mathbb{N}$ such
that $\mathrm{lct}(X)=\mathrm{lct}_{n}(X)$.
\end{conjecture}

The proof of Theorem~\ref{theorem:Tian} uses (at least implicitly)
the~following result.

\begin{theorem}[{\cite{Ti87}, \cite{DeKo01}}]
\label{theorem:KE} The Fano orbifold $X$ is K\"ahler--Einstein if
$$
\mathrm{lct}\big(X\big)>\frac{\mathrm{dim}(X)}{\mathrm{dim}(X)+1}.
$$
\end{theorem}

Note that there are many well-known obstructions to the~existence
of K\"ahler--Einstein~metrics on~smooth Fano manifolds and Fano
orbifolds (see \cite{Mat57}, \cite{Fu83}, \cite{GaMaSpaYau06},
\cite{RossThomas}).

\begin{example}
\label{example:P123} If $X\cong\mathbb{P}(1,2,3)$, then $X$ is not
K\"ahler--Einstein (see \cite{GaMaSpaYau06}, \cite{RossThomas}).
\end{example}

Let us describe one more $\alpha$-invariant that took its origin
in \cite{Tian90}.

Let $\mathcal{M}$ be a~linear system on the~variety $X$. Then
the~number
$$
\mathrm{c}\big(X,\mathcal{M}\big)=\mathrm{sup}\Big\{\epsilon\in\mathbb{Q}\ \Big|\text{the~log pair}\ \big(X, \epsilon \mathcal{M}\big)\ \text{is log canonical}\Big\}\in\mathbb{Q}\cup\big\{+\infty\big\}.%
$$
is called the~log canonical threshold of the~linear system
$\mathcal{M}$ (cf. \cite[Theorem~4.8]{Ko97}). Put
$$
\mathrm{lct}_{n,2}\big(X\big)=\mathrm{inf}\Bigg\{\mathrm{c}\Bigg(X,\frac{1}{n}\mathcal{B}\Bigg)\ \Bigg|\ \mathcal{B}\ \text{is a~pencil in $\big|-nK_{X}\big|$}\Bigg\}%
$$
for every $n\in\mathbb{N}$. The number  $\mathrm{lct}_{n,2}(X)$ is
denoted by $\alpha_{n,2}(X)$ in \cite{ChenWang09B} and
\cite{Wang10}. Note that
\begin{equation}
\label{equation:lct-inf-lct}
\mathrm{lct}\big(X\big)=\mathrm{inf}\Big\{\mathrm{lct}_{n,2}\big(X\big)\ \Big|\ n\in\mathbb{N}\Big\},%
\end{equation}
and it follows from \cite[Theorem~4.8]{Ko97} that
$\mathrm{lct}_{n}(X)\leqslant\mathrm{lct}_{n,2}(X)$ for every $n\in\mathbb{N}$.

\begin{remark}
\label{remark:lct-lct-n-2} It follows from
\cite[Lemma~4.8]{Mustata02} and \cite[Theorem~4.8]{Ko97} that
the~set
$$
\Bigg\{\mathrm{c}\Bigg(X,\frac{1}{n}\mathcal{B}\Bigg)\ \Bigg|\ \mathcal{B}\ \text{is a~pencil in $\big|-nK_{X}\big|$}\Bigg\}%
$$
is finite. Thus, there is a~pencil $\mathcal{B}$ in $|-nK_{X}|$
such that the~equality
$\mathrm{lct}_{n,2}(X)=\mathrm{c}(X,\mathcal{B}/n)$.~Then
$$
\mathrm{lct}_{n,2}\big(X\big)>\mathrm{lct}\big(X\big)
$$
if there exists at most finitely many effective $\mathbb{Q}$-divisors
$D_{1},D_{2},\ldots,D_{r}$ on the~variety $X$ such~that
$$
\mathrm{c}\big(X,D_{1}\big)=\mathrm{c}\big(X,D_{2}\big)=\cdots=\mathrm{c}\big(X,D_{r}\big)=\mathrm{lct}\big(X\big)
$$
and $D_{1}\sim_{\mathbb{Q}} D_{2}\sim_{\mathbb{Q}}\ldots\sim_{\mathbb{Q}}
D_{r}\sim_{\mathbb{Q}} -K_{X}$.
\end{remark}

The importance of the~number $\mathrm{lct}_{n,2}(X)$ is due to
the~following conjecture.

\begin{conjecture}[{cf. \cite[Theorem~2]{ChenWang09B}, \cite[Theorem~1]{Wang10}}]
\label{conjecture:KE-conjecture} Suppose that
$$
\mathrm{lct}_{n,2}\big(X\big)>\frac{\mathrm{dim}(X)}{\mathrm{dim}(X)+1}.
$$
for every $n\in\mathbb{N}$. Then $X$ is K\"ahler--Einstein.
\end{conjecture}

Note that Conjecture~\ref{conjecture:KE-conjecture} is not much
stronger than Theorem~\ref{theorem:KE} by
$(\ref{equation:lct-inf-lct})$.

\begin{example}
\label{example:Cheltsov-Park} Suppose that $X$ is a~smooth
hypersurface in $\mathbb{P}^{m}$ of degree $m\geqslant 3$. Then
$$
\mathrm{lct}_{n}\big(X\big)\geqslant 1-\frac{1}{m}=\frac{\mathrm{dim}(X)}{\mathrm{dim}(X)+1}%
$$
for every $n\in\mathbb{N}$ by \cite{Ch01}. The~equality
$\mathrm{lct}_{n}(X)=1-1/m$ holds $\iff$ the~hypersurface $X$
contains a~cone of dimension $m-2$ (see \cite[Theorem~1.3]{Ch01},
\cite[Theorem~4.1]{Ch01}, \cite[Theorem~0.2]{FEM}). Then
$$
\mathrm{lct}_{n,2}\big(X\big)>\frac{\mathrm{dim}(X)}{\mathrm{dim}(X)+1}
$$
by Remark~\ref{remark:lct-lct-n-2}, \cite[Remark~1.6]{Ch01},
\cite[Theorem~4.1]{Ch01}, \cite[Theorem~5.2]{Ch01} and \cite[Theorem~0.2]{FEM},
because $X$ contains at most finitely many cones by
\cite[Theorem~4.2]{CoolsCoppens}. If $X$ is general, then
$$
1=\mathrm{lct}_{1}\big(X\big)\geqslant\mathrm{lct}\big(X\big)\geqslant\left\{%
\aligned
&3/4\ \text{if $m=3$},\\%
&7/9\ \text{if $m=4$},\\%
&5/6\ \text{if $m=5$},\\%
&1\ \text{if $m\geqslant 5$},\\%
\endaligned\right.%
$$
by \cite{Pu04d}, \cite{Ch07b}, \cite{CheltsovParkWon}. Thus, if
$X$ is general, then it is K\"ahler--Eisntein by
Theorem~\ref{theorem:KE}.
\end{example}

The assertion of Conjecture~\ref{conjecture:KE-conjecture} follows
from \cite[Theorem~2]{ChenWang09B} and \cite[Theorem~1]{Wang10}
under an~additional assumption that the~K\"ahler-Ricci flow on $X$
is tamed (see \cite{ChenWang09B} and \cite{Wang10}).

\begin{theorem}[{\cite{ChenWang09B}, \cite{Wang10}}]
\label{theorem:Chen-Wang} If $\mathrm{dim}(X)=2$, then
the~K\"ahler-Ricci~flow~on~$X$~is~tamed.
\end{theorem}

\begin{corollary}
\label{corollary:Chen-Wang} Suppose that $\mathrm{dim}(X)=2$ and
$$
\mathrm{lct}_{n,2}\big(X\big)>\frac{2}{3}
$$
for every $n\in\mathbb{N}$. Then $X$ is K\"ahler--Einstein.
\end{corollary}

Two-dimensional Fano orbifolds are called del Pezzo surfaces.

\begin{remark}
\label{remark:del-Pezzo} Del Pezzo surfaces with quotient
singularities are not classified (cf. \cite{KeelMcKernan}). But
\begin{itemize}
\item del Pezzo surfaces with canonical singularities are classified (see \cite{HidakaWatanabe}),%

\item del Pezzo surfaces with $2$-Gorenstein quotient singularities are classified (see \cite{AlexeevNikulin}),%

\item del Pezzo surfaces of Picard rank $1$ with $T$-singularities are classified  (see \cite{HackingProkhorov}).%
\end{itemize}
\end{remark}

Del Pezzo surfaces with canonical singularities form a~very
natural class of del Pezzo surfaces.

\begin{problem}
\label{problem:del-Pezzo} Describe all K\"ahler--Einstein del
Pezzo surface with canonical singularities.
\end{problem}

Recall that if $X$ is a~del Pezzo surface with canonical
singularities, then
\begin{itemize}
\item either the~inequality $K_{X}^{2}\geqslant 5$ holds,%
\item or one of the~following possible cases occurs:
\begin{itemize}
\item the~equality $K_{X}^{2}=1$ holds and $X$ is a~sextic surface in $\mathbb{P}(1,1,2,3)$,%
\item the~equality $K_{X}^{2}=2$ holds and $X$ is a~quartic surface in $\mathbb{P}(1,1,1,2)$,%
\item the~equality $K_{X}^{2}=3$ holds and $X$ is a~cubic surface in $\mathbb{P}^{3}$,%
\item the~equality $K_{X}^{2}=4$ holds and $X$ is a~complete intersection in $\mathbb{P}^{4}$ of two quadrics.%
\end{itemize}
\end{itemize}

Let us consider few examples to illustrate the~expected answer to
Problem~\ref{problem:del-Pezzo}.

\begin{example}
\label{example:degree-one} Suppose that $X$ is a~sextic surface in
$\mathbb{P}(1,1,2,3)$ such that its singular locus consists of
singular points of type $\mathbb{A}_{1}$ or $\mathbb{A}_{2}$.
Arguing as in the~proof of \cite[Lemma~4.1]{Ch07b},~we~see~that
$$
\mathrm{lct}_{n,2}\big(X\big)>\frac{2}{3}
$$
for every $n\in\mathbb{N}$. Thus, the~surface $X$ is
K\"ahler--Einstein~by~Corollary~\ref{corollary:Chen-Wang}.
\end{example}

\begin{example}
\label{example:degree-two} Suppose that $X$ is a~quartic surface
in $\mathbb{P}(1,1,1,2)$ such that its singular locus consists of
singular points of type $\mathbb{A}_{1}$ or $\mathbb{A}_{2}$. Then
$X$ is K\"ahler--Einstein by \cite[Theorem~2]{Kollar-Ghigi}.
\end{example}

\begin{example}
\label{example:Ding-Tian} Suppose that $X$ is a~cubic surface in
$\mathbb{P}^{3}$ that is not a~cone.~Then
\begin{itemize}
\item if $X$ is smooth, then $X$ is K\"ahler--Einstein by Theorem~\ref{theorem:Tian},%
\item if $\mathrm{Sing}(X)$ consists of one point of type
$\mathbb{A}_{1}$, then it follows from \cite[Theorem~5.1]{Shi09}
that
$$
\mathrm{lct}_{n,2}\big(X\big)>\frac{2}{3}=\mathrm{lct}_{1}\big(X\big)=\mathrm{lct}\big(X\big)
$$
for every $n\in\mathbb{N}$, which implies that $X$ is K\"ahler--Einstein~by~Corollary~\ref{corollary:Chen-Wang},%
\item if the~cubic surface $X$ has a~singular point that is not a~singular point of type $\mathbb{A}_{1}$ or $\mathbb{A}_{2}$, then the~surface $X$ is not K\"ahler--Einstein by \cite[Proposition~4.2]{Ding-Tian92}.%
\end{itemize}
\end{example}

\begin{example}
\label{example:degree-four} Suppose that $X$ is a~complete
intersection in $\mathbb{P}^{4}$ of two quadrics. Then
\begin{itemize}
\item if $X$ is smooth, then $X$ is K\"ahler--Einstein by Theorem~\ref{theorem:Tian},%

\item if $X$ is K\"ahler--Einstein, then $X$ has at most singular points of type $\mathbb{A}_{1}$ (see \cite{Je97}),%

\item it follows from \cite{MaMu93} or
\cite[Theorem~44]{Kollar-Ghigi} that $X$ is K\"ahler--Einstein if
it is given by
$$
\sum_{i=0}^{4}x_{i}^{2}=\sum_{i=0}^{4}\lambda_{i}x_{i}^{2}=0\subseteq\mathbb{P}^{4}\cong\mathrm{Proj}\Big(\mathbb{C}[x_{0},\ldots,x_{4}]\Big),
$$
and $X$  has at most singular points of type $\mathbb{A}_{1}$,
where $(\lambda_{0}:\lambda_{1}:\lambda_{2}:\lambda_{3}:
\lambda_{4})\in\mathbb{P}^4$.
\end{itemize}
\end{example}

Keeping in mind Examples~\ref{example:degree-one},
\ref{example:degree-two}, \ref{example:Ding-Tian} and
\ref{example:degree-four}, \cite[Example~1.12]{Ch07c} and
\cite[Table~1]{MiyanishiZhang}, it~is~very natural to expect that
the~following answer to Problem~\ref{problem:del-Pezzo} is true
(cf. Example~\ref{example:P123}).

\begin{conjecture}
\label{conjecture:del-Pezzo} If the~orbifold $X$ is a~del Pezzo
surface with at most canonical singularities, then the~surface $X$
is K\"ahler--Enstein $\iff$ it satisfies one of the~following
conditions:
\begin{itemize}
\item $K_{X}^{2}=1$ and $\mathrm{Sing}(X)$ consists of points of type $\mathbb{A}_{1}$, $\mathbb{A}_{2}$, $\mathbb{A}_{3}$, $\mathbb{A}_{4}$, $\mathbb{A}_{5}$, $\mathbb{A}_{6}$, $\mathbb{A}_{7}$ or $\mathbb{D}_{4}$,%
\item $K_{X}^{2}=2$ and $\mathrm{Sing}(X)$ consists of points of type $\mathbb{A}_{1}$, $\mathbb{A}_{2}$ or $\mathbb{A}_{3}$,%
\item $K_{X}^{2}=3$ and $\mathrm{Sing}(X)$ consists of points of type $\mathbb{A}_{1}$ or $\mathbb{A}_{2}$,%
\item $K_{X}^{2}=4$ and $\mathrm{Sing}(X)$ consists of points of type $\mathbb{A}_{1}$,%
\item the~surface $X$ is smooth and $6\geqslant K_{X}^{2}\geqslant 5$,%
\item either $X\cong\mathbb{P}^{2}$ or $X\cong\mathbb{P}^{1}\times \mathbb{P}^{1}$.%
\end{itemize}
\end{conjecture}

In this paper, we prove the~following result.

\begin{theorem}
\label{theorem:auxiliary}  Suppose that $X$ is a~sextic surface in
$\mathbb{P}(1,1,2,3)$. Then
$$
\mathrm{lct}_{n,2}\big(X\big)>\frac{2}{3}
$$
for every $n\in\mathbb{N}$ if $\mathrm{Sing}(X)$ consists of
points of type $\mathbb{A}_{1}$, $\mathbb{A}_{2}$,
$\mathbb{A}_{3}$, $\mathbb{A}_{4}$, $\mathbb{A}_{5}$ or
$\mathbb{A}_{6}$.
\end{theorem}

\begin{corollary}
\label{corollary:main} Suppose that $X$ is a~sextic surface in
$\mathbb{P}(1,1,2,3)$ such that its singular locus consists of
singular points of type  $\mathbb{A}_{1}$, $\mathbb{A}_{2}$,
$\mathbb{A}_{3}$, $\mathbb{A}_{4}$, $\mathbb{A}_{5}$ or
$\mathbb{A}_{6}$. Then $X$ is K\"ahler--Enstein.
\end{corollary}

It should be pointed out that Corollary~\ref{corollary:main} and
Examples~\ref{example:degree-one}, \ref{example:degree-two},
\ref{example:Ding-Tian}, \ref{example:degree-four} illustrate
a~general philosophy that the~existence of
K\"ahler--Enstein~metrics on Fano orbifolds is~related~to
an~algebro-geometric notion of stability (see
\cite[Theorem~4.1]{Ding-Tian92},~\cite{Ti97},~\cite{Donaldson}).

\begin{remark}
\label{remark:DP1-SING} If  $X$ is a~sextic surface in
$\mathbb{P}(1,1,2,3)$ with canonical singularities, then either
$$
\mathrm{Sing}\big(X\big)\in\left\{\aligned &\mathbb{E}_8, \mathbb{E}_7, \mathbb{E}_7+\mathbb{A}_1, \mathbb{E}_6, \mathbb{E}_6+\mathbb{A}_2,\mathbb{E}_6+\mathbb{A}_1, \mathbb{D}_8, \mathbb{D}_7, \mathbb{D}_6, \mathbb{D}_6+\mathbb{A}_1+\mathbb{A}_1, \mathbb{D}_6+\mathbb{A}_1,\\%
&\mathbb{D}_5, \mathbb{D}_5+\mathbb{A}_3, \mathbb{D}_5+\mathbb{A}_2, \mathbb{D}_5+\mathbb{A}_1+\mathbb{A}_1, \mathbb{D}_5+\mathbb{A}_1, \mathbb{D}_4, \mathbb{D}_4+\mathbb{D}_4, \mathbb{D}_4+\mathbb{A}_3, \mathbb{D}_4+\mathbb{A}_2,\\%
&\mathbb{D}_4+\mathbb{A}_1+\mathbb{A}_1+\mathbb{A}_1+\mathbb{A}_1, \mathbb{D}_4+\mathbb{A}_1+\mathbb{A}_1+\mathbb{A}_1, \mathbb{D}_4+\mathbb{A}_1+\mathbb{A}_1, \mathbb{D}_4+\mathbb{A}_1, \mathbb{A}_8,\\%
& \mathbb{A}_7, \mathbb{A}_7+\mathbb{A}_1, \mathbb{A}_6, \mathbb{A}_6+\mathbb{A}_1, \mathbb{A}_5, \mathbb{A}_5+\mathbb{A}_1, \mathbb{A}_5+\mathbb{A}_1+\mathbb{A}_1, \mathbb{A}_5+\mathbb{A}_2, \mathbb{A}_5+\mathbb{A}_2+\mathbb{A}_1,\\%
&\mathbb{A}_4, \mathbb{A}_4+\mathbb{A}_4, \mathbb{A}_4+\mathbb{A}_3, \mathbb{A}_4+\mathbb{A}_2+\mathbb{A}_1, \mathbb{A}_4+\mathbb{A}_2, \mathbb{A}_4+\mathbb{A}_1+\mathbb{A}_1, \mathbb{A}_4+\mathbb{A}_1,\\%
&\mathbb{A}_3, \mathbb{A}_3+\mathbb{A}_3, \mathbb{A}_3+\mathbb{A}_3+\mathbb{A}_1+\mathbb{A}_1, \mathbb{A}_3+\mathbb{A}_2, \mathbb{A}_3+\mathbb{A}_2+\mathbb{A}_1, \mathbb{A}_3+\mathbb{A}_2+\mathbb{A}_1+\mathbb{A}_1, \\%
&\mathbb{A}_3+\mathbb{A}_1+\mathbb{A}_1+\mathbb{A}_1+\mathbb{A}_1, \mathbb{A}_3+\mathbb{A}_1+\mathbb{A}_1+\mathbb{A}_1, \mathbb{A}_3+\mathbb{A}_1+\mathbb{A}_1, \mathbb{A}_3+\mathbb{A}_1 \\%
\endaligned\right\}
$$
or $\mathrm{Sing}(X)$ consists only of points of type
$\mathbb{A}_{1}$ and $\mathbb{A}_{2}$ (see \cite{Urabe}).
\end{remark}

What is known about $\alpha$-invariants of del Pezzo surfaces with
canonical singularities?

\begin{theorem}[{\cite{Ch07b}}]
\label{theorem:GAFA} If $X$ is a~smooth del Pezzo surface, then
$\mathrm{lct}(X)=\mathrm{lct}_{1}(X)$.
\end{theorem}

\begin{theorem}[{\cite{Ch07b}, \cite{Park-Won}}]
\label{theorem:degree-3-4-5-6-7-8} If $X$ is a~del Pezzo surface
with canonical singularities, then
$$
\mathrm{lct}\big(X\big)=\mathrm{lct}_{1}\big(X\big)
$$
in the~case when $K_{X}^{2}\geqslant 3$.
\end{theorem}

\begin{theorem}[{\cite{Park-Won}}]
\label{theorem:degree-2} If $X$ is a~quartic surface in
$\mathbb{P}(1,1,1,2)$ with canonical singularities, then
$$
\mathrm{lct}\big(X\big)=\left\{\aligned
&\mathrm{lct}_{2}\big(X\big)=1/3\ \text{if $X$ has a~singular point of type $\mathbb{A}_{7}$},\\
&\mathrm{lct}_{2}\big(X\big)=2/5\ \text{if $X$ has a~singular point of type $\mathbb{A}_{6}$},\\
&\mathrm{lct}_{1}\big(X\big)\ \text{in the~remaining cases}.\\
\endaligned
\right.
$$
\end{theorem}

In this paper, we prove the~following result (cf.
Example~\ref{example:degree-one}).

\begin{theorem}
\label{theorem:main} Suppose that $X$ is a sextic surface in
$\mathbb{P}(1,1,2,3)$ with canonical singularities, let
$\omega\colon X\to\mathbb{P}(1,1,2)$ be a natural double cover,
and let $R$ be its branch curve in $\mathbb{P}(1,1,2)$.~Then
$$
\mathrm{lct}\big(X\big)=\left\{\aligned
&\mathrm{lct}_{2}\big(X\big)=1/3\ \text{if $\mathrm{Sing}(X)$ consists of a~ point of type $\mathbb{D}_{8}$},\\
&\mathrm{lct}_{2}\big(X\big)=2/5\ \text{if $\mathrm{Sing}(X)$ consists of a~ point of type $\mathbb{D}_{7}$},\\
&\mathrm{lct}_{3}\big(X\big)=1/2\ \text{if $\mathrm{Sing}(X)$ consists of a~ point of type $\mathbb{A}_{8}$},\\
&\mathrm{lct}_{2}\big(X\big)=1/2\ \text{if $\mathrm{Sing}(X)$ consists of a~point of type $\mathbb{A}_{7}$ and a~point of type $\mathbb{A}_{1}$},\\
&\mathrm{lct}_{2}\big(X\big)=1/2\ \text{if $\mathrm{Sing}(X)$ consists of a point of type $\mathbb{A}_{7}$ and $R$ is reducible},\\
&\mathrm{lct}_{3}\big(X\big)=3/5\ \text{if $X$ has a singular point of type $\mathbb{A}_{7}$ and $R$ is irreducible},\\
&\mathrm{lct}_{2}\big(X\big)=2/3\ \text{if $X$ has a~singular point of type $\mathbb{A}_{6}$},\\
&\mathrm{lct}_{2}\big(X\big)=2/3\ \text{if $X$ has a~singular point of type $\mathbb{A}_{5}$},\\
&\mathrm{lct}_{2}\big(X\big)=\mathrm{min}\big(\mathrm{lct}_{1}\big(X\big),4/5\big)\ \text{if $X$ has a~singular point of type $\mathbb{A}_{4}$},\\
&\mathrm{lct}_{1}\big(X\big)\ \text{in the~remaining cases}.\\
\endaligned
\right.
$$
\end{theorem}

It should be pointed out that if $X$ is a~del Pezzo surface with
at most canonical singularities, then all possible
values~of~the~number $\mathrm{lct}_{1}(X)$ are computed in
\cite{Pa01}, \cite{Park2002}, \cite{Park-Won-PEMS}.

\begin{example}
\label{example:Jihun}If $X$ is a~sextic surface in
$\mathbb{P}(1,1,2,3)$ with canonical singularities, then
\begin{itemize}
\item $\mathrm{lct}_{1}(X)=1/6$ $\iff$ the~surface $X$ has a~singular point of type $\mathbb{E}_8$,%
\item $\mathrm{lct}_{1}(X)=1/4$ $\iff$ the~surface $X$ has a~singular point of type $\mathbb{E}_7$,%
\item $\mathrm{lct}_{1}(X)=1/3$ $\iff$ the~surface $X$ has a~singular point of type $\mathbb{E}_6$,%
\item $\mathrm{lct}_{1}(X)=1/2$ $\iff$ the~surface $X$ has a~singular point of type $\mathbb{D}_{4}$, $\mathbb{D}_{5}$, $\mathbb{D}_{6}$, $\mathbb{D}_{7}$ or $\mathbb{D}_{8}$,%
\item $\mathrm{lct}_{1}(X)=2/3$ $\iff$ the~following two conditions are satisfied:%
\begin{itemize}
\item the~surface $X$ has no singular points of type $\mathbb{D}_{4}$, $\mathbb{D}_{5}$, $\mathbb{D}_{6}$, $\mathbb{D}_{7}$, $\mathbb{D}_{8}$, $\mathbb{E}_{6}$, $\mathbb{E}_{7}$ or $\mathbb{E}_{8}$,%
\item there is a~curve in $|-K_X|$ that has a~cusp at a~point in $\mathrm{Sing}(X)$ of type $\mathbb{A}_2$,%
\end{itemize}
\item $\mathrm{lct}_{1}(X)=3/4$ $\iff$ the~following three conditions are satisfied:%
\begin{itemize}
\item the~surface $X$ has no singular points of type $\mathbb{D}_{4}$, $\mathbb{D}_{5}$, $\mathbb{D}_{6}$, $\mathbb{D}_{7}$, $\mathbb{D}_{8}$, $\mathbb{E}_{6}$, $\mathbb{E}_{7}$ or $\mathbb{E}_{8}$,%
\item there is no curve in $|-K_X|$ that has a~cusp at a~point in $\mathrm{Sing}(X)$ of type $\mathbb{A}_2$,%
\item there is a~curve in $|-K_X|$ that has a~cusp at a~point in $\mathrm{Sing}(X)$ of type $\mathbb{A}_1$,%
\end{itemize}
\item $\mathrm{lct}_{1}(X)=5/6$ $\iff$ the~following three conditions are satisfied:%
\begin{itemize}
\item the~surface $X$ has no singular points of type $\mathbb{D}_{4}$, $\mathbb{D}_{5}$, $\mathbb{D}_{6}$, $\mathbb{D}_{7}$, $\mathbb{D}_{8}$, $\mathbb{E}_{6}$, $\mathbb{E}_{7}$ or $\mathbb{E}_{8}$,%
\item there is no curve in $|-K_X|$ that have a~cusp at a~point in $\mathrm{Sing}(X)$,%
\item there is a~curve in $|-K_X|$ that has a~cusp,%
\end{itemize}
\item $\mathrm{lct}_{1}(X)=1$ $\iff$ there are no cuspidal curves in $|-K_X|$.%
\end{itemize}
\end{example}

A crucial role in the~proofs of both
Theorems~\ref{theorem:main}~and~\ref{theorem:auxiliary} is played
by a~new local inequality that we discovered. This inequality is
a~technical tool, but let us describe it now.

Let $S$ be a~surface, let $D$ be an~arbitrary effective
$\mathbb{Q}$-divisor on the~surface $S$, let $O$ be a~smooth point
of the~surface $S$, let $\Delta_{1}$ and $\Delta_{2}$ be reduced
irreducible curves on $S$~such~that
$$
\Delta_{1}\not\subseteq\mathrm{Supp}\big(D\big)\not\supseteq\Delta_{2},
$$
and the~divisor $\Delta_{1}+\Delta_{2}$ has a~simple normal
crossing singularity at the~smooth point
$O\in\Delta_{1}\cap\Delta_{2}$, let $a_{1}$ and $a_{2}$ be some
non-negative rational~numbers. Suppose that the~log pair
$$
\Big(S,\ D+a_{1}\Delta_{1}+a_{2}\Delta_{2}\Big)
$$
is not Kawamata log terminal at $O$, but $(S,
D+a_{1}\Delta_{1}+a_{2}\Delta_{2})$ is Kawamata log terminal~in
a~punctured neighborhood of the~point $O$.

\begin{theorem}
\label{theorem:I} Let $A,B,M,N,\alpha,\beta$ be non-negative
rational numbers. Then
$$
\mathrm{mult}_{O}\Big(D\cdot\Delta_{1}\Big)\geqslant M+Aa_{1}-a_{2}\ \text{or}\ \mathrm{mult}_{O}\Big(D\cdot\Delta_{2}\Big)\geqslant N+Ba_{2}-a_{1}%
$$
in the~case when the~following conditions are satisfied:
\begin{itemize}
\item the~inequality $\alpha a_{1}+\beta a_{2}\leqslant 1$ holds,%
\item the~inequalities $A(B-1)\geqslant 1\geqslant\mathrm{max}(M,N)$ hold,%
\item the~inequalities $\alpha(A+M-1)\geqslant A^{2}(B+N-1)\beta$ and $\alpha(1-M)+A\beta\geqslant A$ hold,%
\item either the~inequality $2M+AN\leqslant 2$ holds or
$$
\alpha\big(B+1-MB-N\big)+\beta\big(A+1-AN-M\big)\geqslant AB-1.
$$%
\end{itemize}
\end{theorem}

\begin{corollary}
\label{corollary:Dimitra} Suppose that
$$
\frac{2m-2}{m+1}a_{1}+\frac{2}{m+1}a_{2}\leqslant 1
$$
for some integer $m$ such that $m\geqslant 3$. Then
$$
\mathrm{mult}_{O}\Big(D\cdot\Delta_{1}\Big)\geqslant 2a_{1}-a_{2}\ \text{or}\ \mathrm{mult}_{O}\Big(D\cdot\Delta_{2}\Big)\geqslant \frac{m}{m-1}a_{2}-a_{1}.%
$$
\end{corollary}

\begin{proof}
To prove the required assertion, let us put
$$
A=2,\ B=\frac{m}{m-1},\ M=0, N=0,\ \alpha=\frac{2m-2}{m+1},\ \beta=\frac{2}{m+1}a_{2},%
$$
and let us check that all hypotheses of Theorem~\ref{theorem:I}
are satisfied.

We have $\alpha a_{1}+\beta a_{2}\leqslant 1$ by assumption. We
have
$$
A(B-1)=\frac{2}{m-1}\geqslant 1\geqslant 0=\mathrm{max}(M,N),
$$
since $m\geqslant 3$. We have
$$
\alpha(A+M-1)=\frac{2m-2}{m+1}\geqslant \frac{8}{m^2-1}=A^{2}(B+N-1)\beta,%
$$
since $m\geqslant 3$. We have $\alpha(1-M)+A\beta=2\geqslant 2=A$
and $2M+AN=0\leqslant 2$.

Thus, we see that all hypotheses of Theorem~\ref{theorem:I} are
satisfied. Then
$$
\mathrm{mult}_{O}\Big(D\cdot\Delta_{1}\Big)\geqslant  M+Aa_{1}-a_{2}=2a_{1}-a_{2}\ \text{or}\ \mathrm{mult}_{O}\Big(D\cdot\Delta_{2}\Big)\geqslant N+Ba_{2}-a_{1}=\frac{m}{m-1}a_{2}-a_{1}%
$$
by Theorem~\ref{theorem:I}.
\end{proof}

For the~convenience of a~reader, we organize the~paper in
the~following way:
\begin{itemize}
\item in Section~\ref{section:preliminaries}, we collect auxiliary results,%
\item in Section~\ref{section:main-inequality}, we prove Theorem~\ref{theorem:I},%
\item in Sections~\ref{section:cyclic-orbifolds}, we prove Theorem~\ref{theorem:main-single-point},%
\item in Sections~\ref{section:non-cyclic-orbifolds}, we prove Theorems~\ref{theorem:main-single-point-non-cyclic},%
\item in Sections~\ref{section:orbifolds-with-many-singular-points}, we prove Theorems~\ref{theorem:main-many-points}.%
\end{itemize}

By Remark~\ref{remark:DP1-SING}, both Theorems~\ref{theorem:auxiliary} and
\ref{theorem:main} follow from Theorems~\ref{theorem:main-single-point},
\ref{theorem:main-single-point-non-cyclic} and \ref{theorem:main-many-points}.

\section{Preliminaries}%
\label{section:preliminaries}

Let $S$ be a~surface with canonical singularities, and let $D$ be
an effective $\mathbb{Q}$-divisor on $S$. Put
$$
D=\sum_{i=1}^{r}a_{i}D_{i},
$$
where $D_{i}$ is an~irreducible curve, and
$a_{i}\in\mathbb{Q}_{>0}$. We assume that $D_{i}\ne D_{j}\iff i\ne
j$.

Suppose that $(S,D)$ is log canonical, but $(S,D)$ is not Kawamata
log terminal.

\begin{remark}
\label{remark:convexity} Let $\bar{D}$ be an~effective
$\mathbb{Q}$-divisor on the~surface $S$ such that
$$
\bar{D}=\sum_{i=1}^{r}\bar{a}_{i}D_{i}\sim_{\mathbb{Q}} D,%
$$
and the~log pair $(S,\bar{D})$ is log canonical, where
$\bar{a}_{i}$ is a~non-negative rational number. Put
$$
\alpha=\mathrm{min}\Bigg\{\frac{a_{i}}{\bar{a}_{i}}\ \Big\vert\ \bar{a}_{i}\ne 0\Bigg\},%
$$
where $\alpha$ is well defined and $\alpha\leqslant 1$. Then
$\alpha=1\iff D=\bar{D}$. Suppose that  $D\ne \bar{D}$. Put
$$
D^{\prime}=\sum_{i=1}^{r}\frac{a_{i}-\alpha\bar{a}_{i}}{1-\alpha}D_{i},%
$$
and choose $k\in\{1,\ldots,r\}$ such that
$\alpha=a_{k}/\bar{a}_{k}$. Then
$D_{k}\not\subset\mathrm{Supp}(D^{\prime})$ and
$D^{\prime}\sim_{\mathbb{Q}} \bar{D}\sim_{\mathbb{Q}} D$, but
the~log pair $(S,D^{\prime})$ is not Kawamata log terminal.
\end{remark}

Let $\mathrm{LCS}(S,D)$ be the~locus of log canonical
singularities of the~log pair $(S,D)$ (see \cite{ChSh08c}).

\begin{theorem}[{\cite[Theorem~17.4]{Ko91}}]
\label{theorem:connectedness} If $-(K_{S}+D)$ is nef~and~big, then
$\mathrm{LCS}(S,D)$ is connected.
\end{theorem}

Take a~point $P\in\mathrm{LCS}(S,D)$. Suppose that
$\mathrm{LCS}(S,D)$ contains no curves that pass through~$P$.

\begin{lemma}
\label{lemma:adjunction} Suppose that $P\not\in\mathrm{Sing}(S)$
and $P\not\in\mathrm{Sing}(D_{1})$. Then
$$
D_{1}\cdot\Bigg(\sum_{i=2}^{r}a_{i}D_{i}\Bigg)\geqslant\sum_{i=2}^{r}a_{i}\mathrm{mult}_{P}\Big(D_{1}\cdot D_{i}\Big)>1.%
$$
\end{lemma}

\begin{proof}
The log pair $(S, D_{1}+\sum_{i=2}^{r}a_{i}D_{i})$ is not log
canonical at $P$, since $a_{1}<1$. Then
$$
D_{1}\cdot\sum_{i=2}^{r}a_{i}D_{i}\geqslant\sum_{i=2}^{r}a_{i}\mathrm{mult}_{P}\Big(D_{1}\cdot D_{i}\Big)\geqslant\mathrm{mult}_{P}\Bigg(\sum_{i=2}^{r}a_{i}D_{i}\Big\vert_{D_{1}}\Bigg)>1%
$$
by \cite[Theorem~17.6]{Ko91}.
\end{proof}

Let $\pi\colon \bar{S}\to S$ be a~birational morphism, and
$\bar{D}$ is a~proper transform of $D$ via $\pi$. Then
$$
K_{\bar{S}}+\bar{D}+\sum_{i=1}^{s}e_{i}E_{i}\sim_{\mathbb{Q}}\pi^{*}\big(K_{S}+D\big),
$$
where $E_{i}$ is an~irreducible $\pi$-exceptional curve, and
$e_{i}\in\mathbb{Q}$. We assume that $E_{i}=E_{j}\iff i=j$.

Suppose, in addition, that the~birational morphism $\pi$ induces
an~isomorphism
$$
\bar{S}\setminus\Bigg(\bigcup_{i=1}^{s}E_{i}\Bigg)\cong S\setminus P.%
$$

\begin{remark}
\label{remark:log-pull-back} The log pair
$(\bar{S},\bar{D}+\sum_{i=1}^{s}e_{i}E_{i})$ is not Kawamata log
terminal at a~point in $\cup_{i=1}^{s}E_{i}$.
\end{remark}

Suppose that $S$ is singular at $P$, and either $P$ is a~singular
point of type $\mathbb{D}_{n}$ for some $n\in\mathbb{N}_{\geqslant
4}$, or the~point $P$ is a~singular point of type $\mathbb{E}_{m}$
for some $m\in\{6,7,8\}$.

\begin{lemma}
\label{lemma:Prokhorov} Suppose that
$E_{1}^{2}=E_{2}^{2}=\cdots=E_{s}^{2}=-2$. Then $e_{1}=1$ if
$$
E_{1}\cdot\Bigg(\sum_{i=2}^{s}E_{i}\Bigg)=3.
$$
\end{lemma}

\begin{proof}
This follows from \cite[Proposition~2.9]{Pr98plt}, because $(S\ni
P)$ is a~weakly-exceptional singularity (see
\cite[Example~4.7]{Pr98plt}, \cite[Example~3.4]{ChSh09a},
\cite[Theorem~3.15]{ChSh09a}).
\end{proof}

\begin{lemma}
\label{lemma:smooth-points} Suppose that $S$ is a~sextic surface
in $\mathbb{P}(1,1,2,3)$ that has canonical singularities, and
suppose that $D\sim_{\mathbb{Q}}-K_{X}$. Let $\mu$ be a~positive
rational number such that either
$$
\mu<\mathrm{lct}_{1}\big(S\big),
$$
or $\mu=2/3$ and $D$ is not a~curve in $|-K_{X}|$ with a~cusp at
a~point in $\mathrm{Sing}(S)$ of type $\mathbb{A}_{2}$.~Then
$$
\mathrm{LCS}\big(S,\mu D\big)\subseteq\mathrm{Sing}\big(S\big),%
$$
the~locus $\mathrm{LCS}(S,\mu D)$ contains no points of type
$\mathbb{A}_{1}$ or $\mathbb{A}_{2}$, and $|\mathrm{LCS}(S,\mu
D)|\leqslant 1$.
\end{lemma}

\begin{proof}
This follows from Theorem~\ref{theorem:connectedness} and
the~proof of \cite[Lemma~4.1]{Ch07b}.
\end{proof}

Most of the~described results are valid in much more general
settings (cf. \cite{Ko91} and \cite{Ko97}).

\section{Local inequality}
\label{section:main-inequality}

The purpose of this section is to prove Theorem~\ref{theorem:I}.

Let $S$ be a~surface, let $D$ be an~arbitrary effective
$\mathbb{Q}$-divisor on the~surface $S$, let $O$ be a~smooth point
of the~surface $S$, let $\Delta_{1}$ and $\Delta_{2}$ be reduced
irreducible curves on $S$~such~that
$$
\Delta_{1}\not\subseteq\mathrm{Supp}\big(D\big)\not\supseteq\Delta_{2},
$$
and the~divisor $\Delta_{1}+\Delta_{2}$ has a~simple normal
crossing singularity at the~smooth point
$O\in\Delta_{1}\cap\Delta_{2}$, let $a_{1}$ and $a_{2}$ be some
non-negative rational~numbers. Suppose that the~log pair
$$
\Big(S,\ D+a_{1}\Delta_{1}+a_{2}\Delta_{2}\Big)
$$
is not Kawamata log terminal at $O$, but $(S,
D+a_{1}\Delta_{1}+a_{2}\Delta_{2})$ is Kawamata log terminal~in
a~punctured neighborhood of the~point $O$. In particular, we must
have $a_{1}<1$ and $a_{2}<1$.

Let $A,B,M,N,\alpha,\beta$ be non-negative rational numbers such
that
\begin{itemize}
\item the~inequality $\alpha a_{1}+\beta a_{2}\leqslant 1$ holds,%
\item the~inequalities $A(B-1)\geqslant 1\geqslant\mathrm{max}(M,N)$ hold,%
\item the~inequalities $\alpha(A+M-1)\geqslant A^{2}(B+N-1)\beta$ and $\alpha(1-M)+A\beta\geqslant A$ holds,%
\item either the~inequality $2M+AN\leqslant 2$ holds or
$$
\alpha\big(B+1-MB-N\big)+\beta\big(A+1-AN-M\big)\geqslant AB-1.
$$%
\end{itemize}

\begin{lemma}
\label{lemma:2-0} The inequalities $A+M\geqslant 1$ and $B>1$
holds. The inequality
$$
\alpha\big(B+1-MB-N\big)+\beta\big(A+1-AN-M\big)\geqslant AB-1
$$
holds. The inequality $\beta(1-N)+B\alpha\geqslant B$ holds. The
inequalities
$$
\frac{\alpha(2-M)}{A+1}+\frac{\beta(2-N)}{B+1}\geqslant 1%
$$
and $\alpha(2-M)B+\beta(1-N)(A+1)\geqslant B(A+1)$ hold.
\end{lemma}

\begin{proof}
The inequality $B>1$ follows from the~inequality $A(B-1)\geqslant
1$. Then
$$
\frac{\alpha}{A+1}+\frac{\beta}{B+1}\geqslant\frac{\alpha}{A+1}+\frac{\beta}{2B}\geqslant\frac{1}{2}
$$
because $2B\geqslant B+1$. Similarly, we see that $A+M\geqslant
1$, because
$$
\frac{\alpha(A+M-1)}{A^{2}(B+N-1)}\geqslant\beta\geqslant 0
$$
and $B+N-1\geqslant 0$. The inequality
$\beta(1-N)+B\alpha\geqslant B$ follows from the~inequalities
$$
\alpha+\frac{\beta(1-N)}{B}\geqslant\frac{2-M}{A+1}\alpha+\frac{\beta(1-N)}{B}\geqslant 1,%
$$
because $A+1\geqslant 2-M$.

Let us show that the~inequality
$$
\alpha\big(2-M\big)B+\beta\big(1-N\big)\big(A+1\big)\geqslant B\big(A+1\big)%
$$
holds. Let $L_{1}$ be the~line in $\mathbb{R}^{2}$ given by the~
equation
$$
x\big(2-M\big)B+y\big(1-N\big)(A+1)-B(A+1)=0
$$
and let $L_{2}$ be the~line that is given by the~equation
$$
x\big(1-M\big)+Ay-A=0,
$$
where $(x,y)$ are coordinates on $\mathbb{R}^{2}$. Then $L_{1}$
intersects the~line $y=0$ at the~point
$$
\Bigg(\frac{A+1}{2-M},0\Bigg)
$$
and $L_{2}$ intersects the~line $y=0$ at the~point $(A/(1-M),0)$.
But
$$
\frac{A+1}{2-M}<\frac{A}{1-M},
$$
which implies that $\alpha(2-M)B+\beta(1-N)(A+1)\geqslant B(A+1)$
if
$$
A^{2}\beta_{0}\big(B+N-1\big)\geqslant \alpha_{0}\big(A+M-1\big),
$$
where $(\alpha_{0},\beta_{0})$ is the~intersection point of the~
lines $L_{1}$ and $L_{2}$. But
$$
\big(\alpha_{0},\beta_{0}\big)=\Bigg(\frac{A(A+1)(B+N-1)}{\Delta},\
\frac{B(A-1+M)}{\Delta}\Bigg),
$$
where $\Delta=2AB-ABM-A+AM-1+M+NA-NAM+N-NM$. But
$$
A^{2}\Big(B\big(A-1+M\big)\Big)\big(B+N-1\big)\geqslant\Big(A\big(A+1\big)\big(B+N-1\big)\Big)\big(A+M-1\big),%
$$
because $A(B-1)\geqslant 1$, which implies that
$A^{2}\beta_{0}(B+N-1)\geqslant \alpha_{0}(A+M-1)$.

Finally, let us show that the~inequality
$$
\alpha\big(B+1-MB-N\big)+\beta\big(A+1-AN-M\big)\geqslant AB-1
$$
holds. Let $L^{\prime}_{1}$ be the~line in $\mathbb{R}^{2}$ given
by the~equation
$$
x\big(B+1-MB-N\big)+y\beta\big(A+1-AN-M\big)-AB+1=0
$$
where $(x,y)$ are coordinates on $\mathbb{R}^{2}$. Then
$L^{\prime}_{1}$ intersects the~line $y=0$ at the~point
$$
\Bigg(\frac{AB-1}{B+1-MB-N},0\Bigg)
$$
and $L_{2}$ intersects the~line $y=0$ at the~point $(A/(1-M),0)$.
But
$$
\frac{AB-1}{B+1-MB-N}<\frac{A}{1-M},
$$
which implies that $\alpha(B+1-MB-N)+\beta(A+1-AN-M)\geqslant
AB-1$ if
$$
A^{2}\beta_{1}\big(B+N-1\big)\geqslant \alpha_{1}\big(A+M-1\big),
$$
where $(\alpha_{1},\beta_{1})$ is the~intersection point of the~
lines $L^{\prime}_{1}$ and $L_{2}$. Note that
$$
\big(\alpha_{1},\beta_{1}\big)=\Bigg(\frac{A(AB-A-2+NA+M)}{\Delta^{\prime}},\ \frac{A+1-NA-M}{\Delta^{\prime}}\Bigg),%
$$
where $\Delta^{\prime}=AB-1-ABM+AM+2M-NAM-M^2$.

To complete the~proof, it is enough to show that the~inequality
$$
A^{2}\Big(A+1-NA-M\Big)(B+N-1)\geqslant \Big(A(AB-A-2+NA+M)\Big)(A+M-1)%
$$
holds. This inequality is equivalent to the~inequality
$$
\big(2-M\big)\big(A+M-1\big)\geqslant A\big(AN+2M-2)\big(B+N-1\big),%
$$
which is true, because $M\leqslant 1$ and $AN+2M-2\leqslant 0$.
\end{proof}

Let us prove prove Theorem~\ref{theorem:I} by reductio ad
absurdum. Suppose that the~inequalities
$$
\mathrm{mult}_{O}\Big(D\cdot\Delta_{1}\Big)<M+Aa_{1}-a_{2}\ \text{and}\ \mathrm{mult}_{O}\Big(D\cdot\Delta_{2}\Big)<N+Ba_{2}-a_{1}%
$$
hold. Let us show that this assumption leads to a~contradiction.

\begin{lemma}
\label{lemma:2-1} The inequalities $a_{1}>(1-M)/A$ and
$a_{2}>(1-N)/B$ hold.
\end{lemma}

\begin{proof}
It follows from Lemma~\ref{lemma:adjunction} that
$$
M+Aa_{1}-a_{2}>\mathrm{mult}_{O}\Big(D\cdot\Delta_{1}\Big)>1-a_{2},%
$$
which implies that $a_{1}>(1-M)/A$. Similarly, we see that
$a_{2}>(1-N)/B$.
\end{proof}

Put $m_{0}=\mathrm{mult}_{O}(D)$. Then $m_{0}$ is a~positive
rational number.

\begin{remark}
\label{remark:2-21} The inequalities $m_{0}<M+Aa_{1}-a_{2}$ and
$m_{0}<N+Ba_{2}-a_{1}$ hold.
\end{remark}

\begin{lemma}
\label{lemma:2-3} The inequality $m_{0}+a_{1}+a_{2}<2$ holds.
\end{lemma}

\begin{proof}
We know that $m_{0}+a_{1}+a_{2}<M+(A+1)a_{1}$ and
$m_{0}+a_{1}+a_{2}<N+(B+1)a_{2}$. Then
$$
\big(m_{0}+a_{1}+a_{2}\big)\Bigg(\frac{\alpha}{A+1}+\frac{\beta}{B+1}\Bigg)<\alpha a_{1}+\beta a_{2}+\frac{\alpha M}{A+1}+\frac{\beta N}{B+1}\leqslant 1+\frac{\alpha M}{A+1}+\frac{\beta N}{B+1},%
$$
which implies that $m_{0}+a_{1}+a_{2}<2$ by Lemma~\ref{lemma:2-0}.
\end{proof}

Let $\pi_1\colon S_{1}\to S$ be the~blow up of the~point $O$, and
let $F_{1}$ be the~$\pi_{1}$-exceptional curve. Then
$$
K_{S_{1}}+D^{1}+a_{1}\Delta^{1}_{1}+a_{2}\Delta^{1}_{2}+\big(m_{0}+a_{1}+a_{2}-1\big)F_{1}\sim_{\mathbb{Q}}\pi_{1}^{*}\Big(K_{S}+D+a_{1}\Delta_{1}+a_{2}\Delta_{2}\Big),
$$
where $D^{1}$, $\Delta^{1}_{1}$, $\Delta^{1}_{2}$ are proper
transforms of the~divisors $D$, $\Delta_{1}$, $\Delta_{2}$ via
$\pi_{1}$, respectively.~Then
$$
\Big(S_{1},\ D^{1}+a_{1}\Delta^{1}_{1}+a_{2}\Delta^{1}_{2}+\big(m_{0}+a_{1}+a_{2}-1\big)F_{1}\Big)%
$$
is not Kawamata log terminal at some point $O_{1}\in F_{1}$ (see
Remark~\ref{remark:log-pull-back}), where
$m_{0}+a_{1}+a_{2}\geqslant 1$.

\begin{lemma}
\label{lemma:2-4} Either $O_{1}=F_{1}\cap\Delta^{1}_{1}$ or
$O_{1}=F_{1}\cap\Delta^{1}_{2}$.
\end{lemma}

\begin{proof}
Suppose that $O_{1}\not \in\Delta^{1}_{1}\cup\Delta^{1}_{2}$. Then
$m_{0}=D^{1}\cdot F_{1}>1$ by Lemma~\ref{lemma:adjunction}. But
$$
m_{0}\Bigg(\frac{\beta+B\alpha}{AB-1}+\frac{\alpha+A\beta}{AB-1}\Bigg)<\big(M+Aa_{1}-a_{2}\big)\frac{\beta+B\alpha}{AB-1}+\big(N+Ba_{2}-a_{1}\big)\frac{\alpha+A\beta}{AB-1},%
$$
because $m_{0}<M+Aa_{1}-a_{2}$ and $m_{0}<N+Ba_{2}-a_{1}$. On
the~other hand, we have
$$
\big(M+Aa_{1}-a_{2}\big)\frac{\beta+B\alpha}{AB-1}+\big(N+Ba_{2}-a_{1}\big)\frac{\alpha+A\beta}{AB-1}\leqslant 1+\frac{M\beta+MB\alpha+N\alpha+AN\beta}{AB-1},%
$$
because $\alpha a_{1}+\beta a_{2}\leqslant 1$ and $AB-1>0$. But we
already proved that $m_{0}>1$. Thus, we see that
$$
\beta+B\alpha+\alpha+A\beta<AB-1+M\beta+MB\alpha+N\alpha+AN\beta,%
$$
which is impossible by Lemma~\ref{lemma:2-0}.
\end{proof}

\begin{lemma}
\label{lemma:2-5} The inequality $O_{1}\ne
F_{1}\cap\Delta^{1}_{1}$ holds.
\end{lemma}

\begin{proof}
Suppose that $O_{1}=F_{1}\cap\Delta^{1}_{1}$. It follows from
Lemma~\ref{lemma:adjunction} that
$$
M+Aa_{1}-a_{2}-m_{0}>\mathrm{mult}_{O_1}\Big(D^1\cdot\Delta^1_{1}\Big)>1-\big(m_{0}+a_{1}+a_{2}-1\big),
$$
which implies that $a_{1}>(2-M)/(A+1)$. Then
$$
\frac{(2-M)\alpha}{A+1}+\frac{\beta(1-N)}{B}<\alpha a_{1}+\beta a_{2}\leqslant 1,%
$$
because $a_{2}>(1-N)/B$ by Lemma~\ref{lemma:2-1}. Thus, we see
that
$$
\frac{(2-M)\alpha}{A+1}+\frac{\beta(1-N)}{B}<1,
$$
which is impossible by Lemma~\ref{lemma:2-0}.
\end{proof}

Therefore, we see that $O_{1}=F_{1}\cap\Delta^{1}_{2}$.  Then the~
log pair
$$
\Big(S_{1},\ D^{1}+a_{1}\Delta^{1}_{1}+a_{2}\Delta^{1}_{2}+\big(m_{0}+a_{1}+a_{2}-1\big)F_{1}\Big)%
$$
is not Kawamata log terminal at the~point $O_{1}$. We know that
$1>m_{0}+a_{1}+a_{2}-1\geqslant 0$.

We have a~blow up $\pi_{1}\colon S_{1}\to S$. For any
$n\in\mathbb{N}$, consider a~sequence of blow ups
$$
\xymatrix{
&S_{n}\ar@{->}[rr]^{\pi_{n}}&&S_{n-1}\ar@{->}[rr]^{\pi_{n-1}}&&\cdots\ar@{->}[rr]^{\pi_{3}}
&&S_{2}\ar@{->}[rr]^{\pi_{2}}&&S_{1}\ar@{->}[rr]^{\pi_{1}}&& S}
$$
such that $\pi_{i+1}\colon S_{i+1}\to S_{i}$ is a~blow up of the~
point $F_{i}\cap\Delta^{i}_{2}$ for every $i\in\{1,\ldots,n-1\}$,
where
\begin{itemize}
\item we denote by $F_{i}$ the~exceptional curve of the~morphism $\pi_{i}$,%
\item we denote by $\Delta^{i}_{2}$ the~proper transform of the~curve $\Delta_{2}$ on the~surface $S_{i}$.%
\end{itemize}

For every $k\in\{1,\ldots,n\}$ and for every $i\in\{1,\ldots,k\}$,
let $D^{k}$, $\Delta^{k}_{1}$ and $F^{k}_{i}$ be
the~proper~transforms on the~surface $S_{k}$ of the~divisors $D$,
$\Delta_{1}$ and $F_{i}$, respectively.~Then
$$
K_{S_{n}}+D^{n}+a_{1}\Delta^{n}_{1}+a_{2}\Delta^{n}_{2}+\sum_{i=1}^{n}\Bigg(a_{1}+ia_{2}-i+\sum_{j=0}^{i-1}m_{j}\Bigg)F_{i}^n\sim_{\mathbb{Q}}\pi^{*}\Big(K_{S}+D+a_{1}\Delta_{1}+a_{2}\Delta_{2}\Big),
$$
where $\pi=\pi_{n}\circ\cdots\circ\pi_{2}\circ\pi_{1}$ and
$m_{i}=\mathrm{mult}_{O_{i}}(D^{i})$ for every
$i\in\{1,\ldots,n\}$. Then the~log pair
\begin{equation}
\label{equation:log-pair} \Bigg(S_{n},\
D^{n}+a_{1}\Delta^{n}_{1}+a_{2}\Delta^{n}_{2}+\sum_{i=1}^{n}\Bigg(a_{1}+ia_{2}-i+\sum_{j=0}^{i-1}m_{j}\Bigg)F^{n}_{i}\Bigg)
\end{equation}
is not Kawamata log terminal at some point of the~set
$F^{n}_{1}\cup F^{n}_{2}\cup\cdots\cup F^{n}_{n}$ (see
Remark~\ref{remark:log-pull-back}).

Put $O_{k}=F_{k}\cap\Delta^{k}_{2}$ for every
$k\in\{1,\ldots,n\}$.

\begin{lemma}
\label{lemma:2-inductive} For every $i\in\{1,\ldots,n\}$, we have
$$
1>a_{1}+ia_{2}-i+\sum_{j=0}^{i-1}m_{j}\geqslant 0,
$$
and $(\ref{equation:log-pair})$ is Kawamata log terminal at every
point of the~set $(F^{n}_{1}\cup F^{n}_{2}\cup\cdots\cup
F^{n}_{n})\setminus O_{n}$.
\end{lemma}

Since $\mathrm{mult}_{O}(D\cdot\Delta_{2})<N+Ba_{2}-a_{1}$ by
assumption, it follows from Lemma~\ref{lemma:2-inductive} that
$$
N+Ba_{2}-a_{1}>\mathrm{mult}_{O}\Big(D\cdot\Delta_{2}\Big)\geqslant\sum_{i=0}^{n-1}m_{i}\geqslant (n-1)(1-a_2)-a_1,%
$$
which implies that $n\leqslant (N+Ba_2)/(1-a_2)$. On the other
hand, the assertion of Lemma~\ref{lemma:2-inductive}~holds for
arbitrary $n\in\mathbb{N}$. So, taking any $n>(N+Ba_2)/(1-a_2)$,
we obtain a contradiction.

We see that to prove Theorem~\ref{theorem:I}, it is enough to
prove~Lemma~\ref{lemma:2-inductive}.

Let us prove~Lemma~\ref{lemma:2-inductive} by induction on
$n\in\mathbb{N}$. The case $n=1$ is already done.

We may assume that~$n\geqslant 2$. For every
$k\in\{1,\ldots,n-1\}$, we may assume that
$$
1>a_{1}+ka_{2}-k+\sum_{j=0}^{k-1}m_{j}\geqslant 0,
$$
the~singularities of the~log pair
$$
\Bigg(S_{k},\
D^{k}+a_{1}\Delta^{k}_{1}+a_{2}\Delta^{k}_{2}+\sum_{i=1}^{k}\Bigg(a_{1}+ka_{2}-k+\sum_{j=0}^{i-1}m_{j}\Bigg)F^{k}_{i}\Bigg)
$$
are Kawamata log terminal along $(F^{k}_{1}\cup
F^{k}_{2}\cup\cdots\cup F^{k}_{k})\setminus O_{k}$ and not
Kawamata log terminal~at~$O_{k}$.

\begin{lemma}
\label{lemma:2-6} The inequality $a_{2}>(n-N)/(B+n-1)$ holds.
\end{lemma}

\begin{proof}
The singularities of the~log pair
$$
\Bigg(S_{n-1},\
D^{n-1}+a_{2}\Delta^{n-1}_{2}+\Bigg(a_{1}+\big(n-1\big)a_{2}-\big(n-1\big)+\sum_{j=0}^{n-2}m_{j}\Bigg)F^{n-1}_{n-1}\Bigg)
$$
are not Kawamata log terminal at the~point $O_{n-1}$. Then it
follows from Lemma~\ref{lemma:adjunction} that
$$
N+Ba_{2}-a_{1}-\sum_{j=0}^{n-2}m_{j}>\mathrm{mult}_{O_{n-1}}\Big(D^{n-1}\cdot\Delta^{n-1}_{2}\Big)>1-\Bigg(a_{1}+\big(n-1\big)a_{2}-\big(n-1\big)+\sum_{j=0}^{n-2}m_{j}\Bigg),
$$
which implies that $a_{2}>(n-N)/(B+n-1)$.
\end{proof}

\begin{lemma}
\label{lemma:2-7} The inequalities
$1>a_{1}+na_{2}-n+\sum_{j=0}^{n-1}m_{j}\geqslant 0$ hold.
\end{lemma}

\begin{proof}
The inequality $a_{1}+na_{2}-n+\sum_{j=0}^{n-1}m_{j}\geqslant 0$
follows from the~fact that the~log pair
$$
\Bigg(S_{n-1},\
D^{n-1}+a_{2}\Delta^{n-1}_{2}+\Bigg(a_{1}+\big(n-1\big)a_{2}-\big(n-1\big)+\sum_{j=0}^{n-2}m_{j}\Bigg)F^{n-1}_{n-1}\Bigg)
$$
is not Kawamata log terminal at the~point $O_{n-1}$.

Suppose that $a_{1}+na_{2}-n+\sum_{j=0}^{n-1}m_{j}\geqslant 1$.
Let us derive a~contradiction.

It follows from Remark~\ref{remark:2-21} that
$m_{0}+a_{2}\leqslant M+Aa_{1}$. Then
$$
a_{1}+nM+nAa_{1}-n\geqslant a_{1}+na_{2}-n+nm_{0}\geqslant a_{1}+na_{2}-n+\sum_{j=0}^{n-1}m_{j}\geqslant 1,%
$$
which implies that $a_{1}\geqslant (n+1-Mn)/(nA+1)$. But
$a_{2}>(n-N)/(B+n-1)$ by Lemma~\ref{lemma:2-6}.~Then
$$
\Bigg(\frac{\alpha(1-M)}{A}+\beta\Bigg)+\alpha\frac{A-1+M}{A(An+1)}+\beta\frac{1-B-N}{B+n-1}=\alpha\frac{n+1-Mn}{nA+1}+\beta\frac{n-N}{B+n-1}<\alpha a_{1}+\beta a_{2}\leqslant 1,%
$$
where $\alpha(1-M)/A+\beta\geqslant 1$ by assumption. Therefore,
we see that
$$
\alpha\frac{A+M-1}{A(An+1)}<\beta\frac{B+N-1}{B+n-1},
$$
where $n\geqslant 2$. But $A+M>1$ and $B+N>1$ by
Lemma~\ref{lemma:2-1}, since $a_{1}<1$ and $a_{2}<1$. Then
$$
\frac{A(An+1)}{\alpha(A+M-1)}>\frac{B+n-1}{\beta(B+N-1)},
$$
but $A^{2}(B+N-1)\beta\leqslant\alpha(A+M-1)$ by assumption. Then
$$
\frac{A}{\alpha(A+M-1)}-\frac{B-1}{\beta(B+N-1)}\geqslant\Bigg(\frac{A^{2}}{\alpha(A+M-1)}-\frac{1}{\beta(B+N-1)}\Bigg)n+\frac{A}{\alpha(A+M-1)}-\frac{B-1}{\beta(B+N-1)}>0,
$$
which implies that $\beta A(B+N-1)>\alpha(B-1)(A+M-1)$. Then
$$
\frac{\alpha(A+M-1)}{A}\geqslant\beta A\big(B+N-1\big)>\alpha\big(B-1\big)\big(A+M-1\big),%
$$
because $A^{2}(B+N-1)\beta\leqslant\alpha(A+M-1)$ by assumption.
Then we  have $\alpha\ne 0$ and $A(B-1)<1$, which is impossible,
because $A(B-1)\geqslant 1$ by assumption.
\end{proof}

\begin{lemma}
\label{lemma:2-8} The log pair~$(\ref{equation:log-pair})$ is
Kawamata log terminal at every point of the~set
$$
F_{n}\setminus\Bigg(\Big(F_{n}\cap
F^{n}_{n-1}\Big)\bigcup\Big(F_{n}\cap\Delta^{n}_{2}\Big)\Bigg).
$$
\end{lemma}

\begin{proof}
Suppose that there is a~point $Q\in F_{n}$ such that
$$
F_{n}\cap F^{n}_{n-1}\ne Q\ne F_{n}\cap\Delta^{n}_{2},
$$
but $(\ref{equation:log-pair})$ is not Kawamata log terminal at
the~point $Q$. Then the~log pair
$$
\Bigg(S_{n},\
D^{n}+\Bigg(a_{1}+na_{2}-n+\sum_{j=0}^{n-1}m_{j}\Bigg)F_{n}\Bigg)%
$$
is not Kawamata log terminal at the~point $Q$ as well. Then
$$
m_{0}\geqslant m_{n-1}=D^{n}\cdot F_{n}>1
$$
by Lemma~\ref{lemma:adjunction}, because
$a_{1}+na_{2}-n+\sum_{j=0}^{n-1}m_{j}<1$ by Lemma~\ref{lemma:2-7}.
Then
$$
m_{0}\Bigg(\frac{\beta+B\alpha}{AB-1}+\frac{\alpha+A\beta}{AB-1}\Bigg)<\big(M+Aa_{1}-a_{2}\big)\frac{\beta+B\alpha}{AB-1}+\big(N+Ba_{2}-a_{1}\big)\frac{\alpha+A\beta}{AB-1},%
$$
because $m_{0}<M+Aa_{1}-a_{2}$ and $m_{0}<N+Ba_{2}-a_{1}$ by
Remark~\ref{remark:2-21}. We have
$$
\big(M+Aa_{1}-a_{2}\big)\frac{\beta+B\alpha}{AB-1}+\big(N+Ba_{2}-a_{1}\big)\frac{\alpha+A\beta}{AB-1}\leqslant 1+\frac{M\beta+MB\alpha+N\alpha+AN\beta}{AB-1},%
$$
because $\alpha a_{1}+\beta a_{2}\leqslant 1$ and $AB-1>0$. But
$m_{0}>1$. Thus, we see that
$$
\beta+B\alpha+\alpha+A\beta<AB-1+M\beta+MB\alpha+N\alpha+AN\beta,%
$$
which contradicts our initial assumptions.
\end{proof}

\begin{lemma}
\label{lemma:2-9} The log pair~$(\ref{equation:log-pair})$ is
Kawamata log terminal at the~point $F_{n}\cap F^{n}_{n-1}$.
\end{lemma}

\begin{proof}
Suppose that $(\ref{equation:log-pair})$ is not Kawamata log
terminal at $F_{n}\cap F^{n}_{n-1}$. Then the~log pair
$$
\Bigg(S_{n},\
D^{n}+\Bigg(a_{1}+\big(n-1\big)a_{2}-\big(n-1\big)+\sum_{j=0}^{n-2}m_{j}\Bigg)F^{n}_{n-1}+\Bigg(a_{1}+na_{2}-n+\sum_{j=0}^{n-1}m_{j}\Bigg)F_{n}\Bigg)%
$$
is not Kawamata log terminal at the~point $F_{n}\cap F^{n}_{n-1}$
as well. Then
$$
m_{n-2}-m_{n-1}=D^{n}\cdot
F_{n-2}>1-\Bigg(a_{1}+na_{2}-n+\sum_{j=0}^{n-1}m_{j}\Bigg)
$$
by Lemma~\ref{lemma:adjunction}, because
$a_{1}+(n-1)a_{2}-(n-1)+\sum_{j=0}^{n-2}m_{j}<1$. Note that
$$
M+Aa_{1}-a_{2}-m_{0}>\mathrm{mult}_{O}\Big(D\cdot\Delta_{1}\Big)-m_{0}\geqslant \mathrm{mult}_{O}\big(D\big)\mathrm{mult}_{O}\big(\Delta_{1}\big)-m_{0}=0,%
$$
which implies that $m_{0}+a_{2}<Aa_{1}+M$. Then
$$
nM+nAa_{1}-na_{2}>nm_{0}\geqslant m_{n-2}-m_{n-1}+\sum_{j=0}^{n-1}m_{j}>n+1-a_{1}-na_{2},%
$$
which gives $a_{1}>(n+1-nM)/(An+1)$.

Now arguing as in the~proof of Lemma~\ref{lemma:2-7}, we obtain
a~contradiction.
\end{proof}

The assertion of Lemma~\ref{lemma:2-inductive} is proved. The
assertion of Theorem~\ref{theorem:I} is proved.

\section{One cyclic singular point}
\label{section:cyclic-orbifolds}

Let $X$ be a~sextic surface in $\mathbb{P}(1,1,2,3)$ with
canonical singularities such that $|\mathrm{Sing}(X)|=1$,
let~$\omega\colon X\to\mathbb{P}(1,1,2)$ be the natural double
cover, let $R$ be its ramification curve in $\mathbb{P}(1,1,2)$,
and~suppose that $\mathrm{Sing}(X)$ consists of one singular point
of type $\mathbb{A}_{m}$, where $m\in\{1,\ldots,8\}$.

\begin{theorem}
\label{theorem:main-single-point} The following equality holds:
$$
\mathrm{lct}\big(X\big)=\left\{\aligned
&\mathrm{lct}_{3}\big(X\big)=1/2\ \text{if $m=8$},\\
&\mathrm{lct}_{2}\big(X\big)=1/2\ \text{if $m=7$ and $R$ is reducible},\\
&\mathrm{lct}_{3}\big(X\big)=3/5\ \text{if $m=7$ and $R$ is irreducible},\\
&\mathrm{lct}_{2}\big(X\big)=2/3\ \text{if $m=6$},\\
&\mathrm{lct}_{2}\big(X\big)=2/3\ \text{if $m=5$},\\
&\mathrm{lct}_{2}\big(X\big)=4/5\ \text{if $m=4$},\\
&\mathrm{lct}_{1}\big(X\big)\ \text{in the~remaining cases},\\
\endaligned
\right.
$$
and if $\mathrm{lct}(X)=2/3$, then there is a~unique effective
$\mathbb{Q}$-divisor $D$ on  $X$ such that $D\sim_{\mathbb{Q}}
-K_{X}$~and
$$
\mathrm{c}\big(X,D\big)=\mathrm{lct}\big(X\big)=\frac{2}{3}.
$$
\end{theorem}

By Theorem~\ref{theorem:KE}, Corollary~\ref{corollary:Chen-Wang}
and Remark~\ref{remark:lct-lct-n-2}, we obtain the~following two
corollaries.

\begin{corollary}
\label{corollary:auxiliary-single-point}  If $m\leqslant 6$, then
$\mathrm{lct}_{n,2}(X)>2/3$ for every $n\in\mathbb{N}$.
\end{corollary}

\begin{corollary}
\label{corollary:main-single-point} If $m\leqslant 6$, then $X$ is
K\"ahler--Enstein.
\end{corollary}

In the~rest of this section we will prove
Theorem~\ref{theorem:main-single-point}.

Let $D$ be an~arbitrary effective $\mathbb{Q}$-divisor on
the~surface $X$~such~that
$$
D\sim_{\mathbb{Q}} -K_{X},
$$
and put $\mu=\mathrm{c}(X,D)$. To prove
Theorem~\ref{theorem:main-single-point}, it is enough to show that
$$
\mu\geqslant\left\{\aligned
&\mathrm{lct}_{3}\big(X\big)=1/2\ \text{if $m=8$},\\
&\mathrm{lct}_{2}\big(X\big)=1/2\ \text{if $m=7$ and $R$ is reducible},\\
&\mathrm{lct}_{3}\big(X\big)=3/5\ \text{if $m=7$ and $R$ is irreducible},\\
&\mathrm{lct}_{2}\big(X\big)=2/3\ \text{if $m=6$},\\
&\mathrm{lct}_{2}\big(X\big)=2/3\ \text{if $m=5$},\\
&\mathrm{lct}_{2}\big(X\big)=4/5\ \text{if $m=4$},\\
&\mathrm{lct}_{1}\big(X\big)\ \text{in the~remaining cases},\\
\endaligned
\right.
$$
and if $\mu=\mathrm{lct}(X)=2/3$, then $D$ is uniquely defined. Note that
$\mathrm{lct}_{1}(X)\geqslant 5/6$ if $m\geqslant
3$~(see~\cite{Park-Won-PEMS}).

Let us prove Theorem~\ref{theorem:main-single-point}. By
Lemma~\ref{lemma:smooth-points}, we may assume~that~$m\geqslant 3$ and
$\mu<\mathrm{lct}_{1}(X)$.~Then
$$
\mathrm{LCS}\big(X,\mu D\big)=\mathrm{Sing}\big(X\big)
$$
by Lemma~\ref{lemma:smooth-points}. Put $P=\mathrm{Sing}(X)$.

Let $\pi\colon\bar{X}\to X$ be a~minimal resolution, let
$E_{1},E_{2},\ldots,E_{m}$ be $\pi$-exceptional curves such~that
$$
E_{i}\cdot E_{j}\ne 0\iff \big|i-j\big|\leqslant 1,
$$
let $C$ be the~curve in $|-K_{X}|$ such that $P\in C$, and let
$\bar{C}$ be it~proper transform on $\bar{X}$.~Then
$$
\bar{C}\sim_{\mathbb{Q}}\pi^*\big(C\big)-\sum_{i=1}^{m}E_i,
$$
and the~curve $C$ is irreducible. We may assume that $D\ne C$,
because $\mu\geqslant\mathrm{lct}_{1}(X)$ if $D=C$.

By Remark~\ref{remark:convexity}, we may assume that
$C\not\subset\mathrm{Supp}(D)$.

Let $\bar{D}$ be the~proper transform of the~divisor $D$ on
the~surface $\bar{X}$. Then
$$
\bar{D}\sim_{\mathbb{Q}}\pi^*\big(D\big)-\sum_{i=1}^{m}a_iE_i,
$$
where $a_{i}$ is a~non-negative rational number. Then the~log pair
\begin{equation}
\label{equation:log-pull-back}
\Big(\bar{X},\mu\bar{D}+\sum_{i=1}^{m}\mu a_iE_i\Big)
\end{equation}
is not Kawamata log terminal (by
Remark~\ref{remark:log-pull-back}). On the~other hand, we have
$$
\bar{D}\cdot E_{1}=2a_1 - a_2,\ \bar{D}\cdot E_{2}=2a_2-a_1-a_3,\ \cdots,\ \bar{D}\cdot E_{m-1}=2a_{m-1}-a_{m-2}-a_m,\ \bar{D}\cdot E_{m}=2a_m - a_{m-1},%
$$
where all intersections $\bar{D}\cdot E_{1}, \bar{D}\cdot E_{2},
\ldots, \bar{D}\cdot E_{m}$ are non-negative. Moreover, we have
$$
\bar{D}\cdot\bar{C}=1-a_1-a_m,
$$
where the intersection
$\bar{D}\cdot\bar{C}$ is non-negative, since
$C\not\subset\mathrm{Supp}(D)$ by assumption. Hence, we have
\begin{equation}
\label{equation:cyclic-orbifolds}
\left\{\aligned
&a_1\geqslant \frac{a_2}{2},\\
&a_2\geqslant \frac{a_1+a_3}{2},\\
&a_3\geqslant \frac{a_2+a_4}{2},\\
&\cdots\\
&a_{m-1}\geqslant \frac{a_{m-2}+a_m}{2},\\
&a_m\geqslant\frac{a_{m-1}}{2},\\
&1\geqslant a_1+a_m.\\
\endaligned
\right.
\end{equation}

It should be pointed out that at least one inequality in
$(\ref{equation:cyclic-orbifolds})$ must be strict, since
$\bar{D}\cdot E_{i}>0$ for at least one $i\in\{1,\ldots,m\}$,
because $P\in\mathrm{Supp}(D)$. Then $a_i>0$ for some
$i\in\{1,\ldots,m\}$.

Note that $a_{1}\geqslant a_{2}/2$ by
$(\ref{equation:cyclic-orbifolds})$. Similarly, it follows from
$(\ref{equation:cyclic-orbifolds})$ that
$$
a_2\geqslant \frac{a_1+a_3}{2}\geqslant\frac{a_1}{4}+\frac{a_3}{4},%
$$
which implies that $a_2\geqslant 2a_3/3$. Arguing in the same way,
we see that
$$
a_{k}\geqslant \frac{k}{k+1}a_{k+1}
$$
for every $k\in\{1,\ldots,m-1\}$ (use
$(\ref{equation:cyclic-orbifolds})$ and induction on $k$). Using
symmetry, we see that
$$
a_{k+1}\geqslant \frac{m-k}{m-k+1}a_{k}
$$
for every $k\in\{1,\ldots,m-1\}$. In particular, the inequality
$a_k>0$ holds for every $k\in\{1,\ldots,m\}$, since we already
know that $a_i>0$ for some $i\in\{1,\ldots,m\}$.

\begin{lemma}
\label{lemma:cyclic-orbifolds-blow-up} Suppose that $\mu a_{i}<1$
for every $i\in\{1,\ldots,m\}$. Then
\begin{itemize}
\item there exists a~point
$$
Q\in\Big\{E_{1}\cap E_{2}, E_{2}\cap E_{3}, \ldots, E_{m-1}\cap E_{m}\Big\}%
$$
such~that the~log pair $(\ref{equation:log-pull-back})$ is not Kawamata log terminal at $Q$,%

\item the~log pair $(\ref{equation:log-pull-back})$ is Kawamata log terminal outside of the~point~$Q$,%

\item if $\mu<(m+1)/(2m-2)$, then $Q\ne E_{1}\cap E_{2}$ and $Q\ne E_{m-1}\cap E_{m}$.%
\end{itemize}
\end{lemma}

\begin{proof}
It follows from Remark~\ref{remark:log-pull-back} and
Theorem~\ref{theorem:connectedness} that there is a~point $Q\in
\cup_{i=1}^{m}E_{i}$ such that the~log pair
$(\ref{equation:log-pull-back})$ is not Kawamata log terminal at
$Q$ and is Kawamata log terminal elsewhere.

Suppose that $Q\in E_1$ and $Q\not\in E_{2}$. Then
$$
2a_1-a_{2}=\bar{D}\cdot E_i>1
$$
by Lemma~\ref{lemma:adjunction}. Taking
$(\ref{equation:cyclic-orbifolds})$ into account, we get
$$
\left\{\aligned
&a_1>\frac{1}{2}+\frac{a_2}{2},\\
&a_2\geqslant \frac{a_1+a_3}{2},\\
&a_3\geqslant \frac{a_2+a_4}{2},\\
&\cdots\\
&a_{m-1}\geqslant \frac{a_{m-2}+a_m}{2},\\
&a_m\geqslant\frac{a_{m-1}}{2},\\
\endaligned
\right.
$$
and adding all these inequalities together we get
$$
\sum_{i=1}^{m}a_i>\frac{1}{2}+\frac{a_1}{2}+\sum_{i=2}^{m-1}a_i+\frac{a_m}{2},
$$
which implies that $a_1+a_m>1$. However, the later is impossible,
since $a_1+a_m\leqslant 1$ by $(\ref{equation:cyclic-orbifolds})$.

We see that if $Q\in E_1$, then $Q=E_{1}\cap E_{2}$. Similarly, we
see that $Q=E_{m-1}\cap E_{m}$ if $Q\in E_m$.

Suppose that $Q\in E_i$ and $Q\not\in E_{j}$ for every  $j\ne i$.
Then $i\ne 1$ and $i\ne m$. We have
$$
2a_i-a_{i-1}-a_{i+1}=\bar{D}\cdot E_i>1
$$
by Lemma~\ref{lemma:adjunction}. Taking
$(\ref{equation:cyclic-orbifolds})$ into account, we get
$$
\left\{\aligned
&a_1>\frac{a_2}{2},\\
&a_2\geqslant \frac{a_1+a_3}{2},\\
&a_3\geqslant \frac{a_2+a_4}{2},\\
&\cdots\\
&a_i\geqslant \frac{1}{2}+\frac{a_{i-1}+a_{i+1}}{2},\\
&\cdots\\
&a_{m-1}\geqslant \frac{a_{m-2}+a_m}{2},\\
&a_m\geqslant\frac{a_{m-1}}{2},\\
\endaligned
\right.
$$
and adding all these inequalities together we get
$$
\sum_{i=1}^{m}a_i>\frac{1}{2}+\frac{a_1}{2}+\sum_{i=2}^{m-1}a_i+\frac{a_m}{2},
$$
which implies that $a_1+a_m>1$. However, the later is impossible,
since $a_1+a_m\leqslant 1$ by $(\ref{equation:cyclic-orbifolds})$.

Thus, we see that there is $k\in\{1,\ldots,m-1\}$ such that
$Q=E_{k}\cap E_{k+1}$.

Suppose that $\mu<(m+1)/(2m-2)$. Let us show that $k\ne 1$ and
$k\ne m-1$.

Due to symmetry, it is enough to show that $k\ne 1$. Recall that
$m\geqslant 3$.

Suppose that $k=1$. Then $Q=E_{1}\cap E_{2}$. Take
$\bar{\mu}\in\mathbb{Q}$ such that
$(m+1)/(2m-2)>\bar{\mu}>\mu$~and
$$
\Big(\bar{X},\mu\bar{D}+\bar{\mu}a_1E_1+\bar{\mu}a_2E_2\Big)
$$
is not Kawamata log terminal at $Q$ and is Kawamata log terminal
outside of the~point $Q$. Then
$$
\frac{2m-2}{m+1}\bar{\mu} a_{1}+\frac{2}{m+1}\bar{\mu}a_{2}<a_{1}+\frac{1}{m-1}a_{2}\leqslant 1,%
$$
by $(\ref{equation:cyclic-orbifolds})$, since $a_1\ne 0$ and
$a_2\ne 0$. On the~other hand, we have
$$
\mathrm{mult}_Q\Big(\mu\bar{D}\cdot E_1\Big)\leqslant \mu\bar{D}\cdot E_1=\mu\Big(2a_1-a_2\Big)<\bar{\mu}\Big(2a_1-a_2\Big),%
$$
since $\mu<\bar{\mu}$. Therefore, it follows from
Corollary~\ref{corollary:Dimitra} that
$$
\mu\Big(2a_2-a_1-a_3\Big)=\mu\bar{D}\cdot E_2\geqslant\mathrm{mult}_{Q}\Big(\mu\bar{D}\cdot E_2\Big)\geqslant\frac{m}{m-1}\bar{\mu}a_2-\bar{\mu}a_1,%
$$
which implies that $a_2(m-2)>a_3(m-1)$, since $\mu<\bar{\mu}$. But
we proved earlier that
$$
a_{3}\geqslant \frac{m-2}{m-1}a_{2},
$$
which is impossible, since $a_2(m-2)>a_3(m-1)$. Thus, we see that
$k\ne 1$.
\end{proof}

If $m=3$, then it follows from $(\ref{equation:cyclic-orbifolds})$
that $a_{1}\leqslant 3/4$, $a_{2}\leqslant 1$, $a_{3}\leqslant
3/4$.

\begin{corollary}
\label{corollary:A3-point} If $m=3$, then
$\mu\geqslant\mathrm{lct}_{1}(X)\geqslant 5/6$.
\end{corollary}

\begin{lemma}
\label{lemma:single-point-A4} Suppose that $m=4$. Then
$\mu\geqslant\mathrm{lct}_{2}(X)=4/5$.
\end{lemma}

\begin{proof}
There is a~unique smooth irreducible curve $\bar{Z}\subset\bar{X}$
such that
$$
\bar{Z}\sim\pi^{*}\big(-2K_{X}\big)-E_1-2E_2-2E_3-E_4
$$
and $E_2\cap E_3\in\bar{Z}$  (cf. the~proof of
Lemma~\ref{lemma:many-points-A4}). Put $Z=\pi(\bar{Z})$. Then
$$
\mathrm{lct}_{2}\big(X\big)\leqslant\mathrm{c}\Big(X, \frac{1}{2}Z\Big)=\frac{4}{5}.%
$$

To complete the~proof, it is enough to show that $\mu\geqslant
4/5$. Suppose that $\mu<4/5$.

By Remark~\ref{remark:convexity}, we may assume that
$Z\not\subset\mathrm{Supp}(D)$, because $Z$ is irreducible.

It follows from $(\ref{equation:cyclic-orbifolds})$ that
$a_{1}\leqslant 4/5$, $a_{2}\leqslant 6/5$, $a_{3}\leqslant 6/5$,
$a_{4}\leqslant 4/5$.

Put $Q=E_2\cap E_3$. Then it follows from
Lemma~\ref{lemma:cyclic-orbifolds-blow-up} that
$(\ref{equation:log-pull-back})$ is not Kawamata log terminal at
the~point~$Q$ and is Kawamata log terminal outside of the~point
$Q$. Then
$$
2a_2-\frac{1}{2}a_2-a_3 \geqslant 2a_2-a_1-a_3=\bar{D}\cdot E_2\geqslant\mathrm{mult}_{Q}\Big(\bar{D}\cdot E_2\Big)>\frac{5}{4}-a_3,%
$$
by Lemma~\ref{lemma:adjunction}. Similarly, we see that
$$
2a_3-a_2-a_4=\bar{D}\cdot E_3\geqslant\mathrm{mult}_Q\Big(\bar{D}\cdot E_3 \Big)>\frac{5}{4}-a_2,%
$$
which implies that $a_2>5/6$ and $a_3>5/6$.

Let $\xi\colon\tilde{X}\to\bar{X}$ be a~blow up of the~point $Q$,
let $E$ be the~exceptional curve of the~blow up $\xi$, and let
$\tilde{D}$ be the~proper transform of the~divisor $\bar{D}$  on
the~surface $\tilde{X}$. Put $\delta=\mathrm{mult}_{Q}(\bar{D})$.

Let $\tilde{E}_{1}$, $\tilde{E}_{2}$, $\tilde{E}_{3}$,
$\tilde{E}_{4}$ be the~proper transforms on $\tilde{X}$  of
$E_{1}$, $E_{2}$, $E_{3}$, $E_{4}$, respectively. Then
\begin{equation}
\label{equation:log-pull-pack-A4}
\Big(\tilde{X},\mu\tilde{D}+\mu a_2\tilde{E}_2+\mu a_3\tilde{E}_3+ \big(\mu a_2+\mu a_3+\mu \delta-1\big)E\Big)%
\end{equation}
is not Kawamata log terminal at some point $O\in E$.

Let $\tilde{Z}$ be the~proper transform on $\tilde{X}$ of the~
curve $\bar{Z}$. Then
$$
0\leqslant\tilde{Z}\cdot\tilde{D}=2-a_2-a_3-\mathrm{mult}_{Q}\big(\bar{D}\big)=2-a_2-a_3-\delta,%
$$
which implies that $\delta+a_{2}+a_{3}\leqslant 2$. We have $\mu
a_2+\mu a_3+\mu \delta-1\leqslant 2\mu-1\leqslant 3/5$, which
implies that~$(\ref{equation:log-pull-pack-A4})$ is Kawamata log
terminal outside of the~point $O$ by
Theorem~\ref{theorem:connectedness}. We have
$$
\left\{\aligned
&2a_3-a_2-a_4-\delta=\tilde{E}_3\cdot\tilde{D}\geqslant 0,\\
&2a_2-a_1-a_3-\delta=\tilde{E}_2\cdot\tilde{D}\geqslant 0,\\
\endaligned
\right.
$$
which implies that $\delta\leqslant 1$. If
$O\not\in\tilde{E}_2\cup\tilde{E}_3$, then
$$
1\geqslant \delta=\tilde{D}\cdot E\geqslant\mathrm{mult}_{O}\Big(\tilde{D}\cdot E\Big)>\frac{5}{4}%
$$
by Lemma~\ref{lemma:adjunction}. Thus, we see that either
$O=\tilde{E}_2\cap E$ or $O=\tilde{E}_3\cap E$.

Without loss of generality, we may assume that $O=\tilde{E}_2\cap
E$. By Lemma~\ref{lemma:adjunction}, one has
$$
\frac{5}{4}-a_2>\frac{7}{6}-a_{2}=2-\frac{5}{6}-a_2>2-a_2-a_3\geqslant \delta=\tilde{D}\cdot E\geqslant\mathrm{mult}_{O}\Big(\tilde{D}\cdot E\Big)>\frac{5}{4}-a_2,%
$$
since $\delta+a_{2}+a_{3}\leqslant 2$ and $a_3>5/6$. The obtained
contradiction concludes the~proof.
\end{proof}

Let $\tau$ be a~biregular involution of the~surface $\bar{X}$ that
is induced by the~double cover $\omega$.

\begin{lemma}
\label{lemma:single-point-A5} Suppose that $m=5$. Then there
exists a unique curve $Z\in |-2K_{X}|$ such that
$$
\mathrm{c}\Bigg(X,\frac{1}{2}Z\Bigg)=\mathrm{lct}_{2}\big(X\big)=\frac{2}{3},
$$
and either $D=Z/2$ or $\mu>2/3$.
\end{lemma}

\begin{proof}
Let $\alpha\colon\bar{X}\to\breve{X}$ be a~contraction of
the~curves $\bar{C}$, $E_5$, $E_4$, $E_3$. Then
$$
\alpha\big(E_{1}\big)\cdot\alpha\big(E_{1}\big)=\alpha\big(E_{2}\big)\cdot\alpha\big(E_{2}\big)=-1,
$$
and $\breve{X}$ is a~smooth del Pezzo surface such that
$K_{\breve{X}}^{2}=5$, which implies that there is a~smooth
irreducible rational curve $\breve{L}_{2}$ on the~surface
$\breve{X}$~such~that $\breve{L}_{2}\cdot\alpha(E_{2})=1$ and
$\breve{L}_{2}\cdot \breve{L}_{2}=-1$.

Let $\bar{L}_{2}$ be the~proper transform of the~curve
$\breve{L}_{2}$ on the~surface $\bar{X}$. Then $\bar{L}_{2}\cdot
\bar{L}_{2}=-1$ and
$$
-K_{\bar{X}}\cdot \bar{L}_{2}=E_{2}\cdot\bar{L}_{2}=1,%
$$
which implies that $E_{1}\cdot \bar{L}_{2}=E_{3}\cdot
\bar{L}_{2}=E_{4}\cdot \bar{L}_{2}=E_{5}\cdot
\bar{L}_{2}=\bar{C}\cdot\bar{L}_{2}=0$.

Let $\beta\colon\bar{X}\to\check{X}$ be a~contraction of
the~curves $\bar{L}_{2}$, $\bar{C}$, $E_5$, $E_4$. Then
$$
\beta\big(E_{2}\big)\cdot\beta\big(E_{2}\big)=\beta\big(E_{3}\big)\cdot\beta\big(E_{3}\big)=-1,
$$
and $\check{X}$ is a~smooth del Pezzo surface such that
$K_{\check{X}}^{2}=5$, which implies that there is an~irreducible
smooth curve $\check{L}_{3}\subset\check{X}$~such~that
$\check{L}_{3}\cdot\beta(E_{3})=1$ and
$\check{L}_{3}\cdot\check{L}_{3}=-1$ (cf. the~proof of
Lemma~\ref{lemma:many-points-A5}).

Let $\bar{L}_{3}$ be the~proper transform of the~curve
$\check{L}_{3}$ on the~surface $\bar{X}$. Then $\bar{L}_{3}\cdot
\bar{L}_{3}=-1$ and
$$
-K_{\bar{X}}\cdot \bar{L}_{3}=E_{3}\cdot \bar{L}_{3}=1,%
$$
which implies that $E_{1}\cdot \bar{L}_{3}=E_{2}\cdot
\bar{L}_{3}=E_{4}\cdot \bar{L}_{3}=E_{5}\cdot
\bar{L}_{3}=\bar{C}\cdot\bar{L}_{3}=0$.

If $\tau(\bar{L}_{3})=\bar{L}_{3}$, then $2\pi(\bar{L}_{3})\sim
-2K_{X}$, but $\pi(\bar{L}_{3})$ is not a~Cartier divisor.

Put $Z=\pi(\bar{L}_{3}+\tau(\bar{L}_{3}))$. Then $Z\sim -2K_{X}$
and $\mathrm{c}(X, Z)=1/3$. We see that
$\mathrm{lct}_{2}(X)\leqslant 2/3$.

Suppose that $D\ne Z/2$. To complete the~proof, it is enough to
show that $\mu>2/3$.

Suppose that $\mu\leqslant 2/3$. Let us derive a~contradiction. It
follows from $(\ref{equation:cyclic-orbifolds})$ that
$$
a_1\leqslant\frac{5}{6},\ a_2\leqslant\frac{4}{3},\ a_3\leqslant\frac{3}{2},\ a_4\leqslant\frac{4}{3},\ a_5\leqslant\frac{5}{6}.%
$$

By Remark~\ref{remark:convexity}, without loss of generality we
may assume that $\pi(\bar{L}_{3})\not\subset\mathrm{Supp}(D)$.
Then
$$
1-a_3=\bar{L}_3\cdot\bar{D}\geqslant 0,
$$
which implies that $a_{3}\leqslant 1$.

Put $Q=E_2\cap E_3$. By
Lemma~\ref{lemma:cyclic-orbifolds-blow-up}, we may assume that
$(\ref{equation:log-pull-back})$ is not Kawamata~log~terminal at
the~point $Q$ and is Kawamata~log~terminal outside of the~point
$Q$. Then
$$
2a_3-a_2-a_4=\bar{D}\cdot E_3\geqslant\mathrm{mult}_{Q}\Big(\bar{D}\cdot E_3\Big)\geqslant\frac{1}{\mu}-a_2>\frac{3}{2}-a_2%
$$
by Lemma~\ref{lemma:adjunction}, which implies that $a_3>9/8$ by
$(\ref{equation:cyclic-orbifolds})$. But $a_3 \leqslant 1$.
\end{proof}

\begin{lemma}
\label{lemma:single-point-A6} Suppose that $m=6$. Then there
exists a unique curve $Z\in |-2K_{X}|$ such that
$$
\mathrm{c}\Bigg(X,\frac{1}{2}Z\Bigg)=\mathrm{lct}_{2}\big(X\big)=\frac{2}{3}
$$
and either $D=Z/2$ or $\mu>2/3$.
\end{lemma}

\begin{proof}
Let $\alpha\colon\bar{X}\to\breve{X}$ be a~contraction of
the~curves $\bar{C}$, $E_{6}$, $E_5$, $E_4$ and $E_3$. Then
$$
\alpha\big(E_{1}\big)\cdot\alpha\big(E_{1}\big)=\alpha\big(E_{2}\big)\cdot\alpha\big(E_{2}\big)=-1,
$$
and $\breve{X}$ is a~smooth del Pezzo surface such that
$K_{\breve{X}}^{2}=6$, which implies that there is a~smooth
irreducible rational curve $\breve{L}_{2}$ on the~surface
$\breve{X}$~such~that $\breve{L}_{2}\cdot\alpha(E_{2})=1$ and
$\breve{L}_{2}\cdot \breve{L}_{2}=-1$.

Let $\bar{L}_{2}$ be the~proper transform of the~curve
$\breve{L}_{2}$ on the~surface $\bar{X}$. Then $\bar{L}_{2}\cdot
\bar{L}_{2}=-1$ and
$$
-K_{\bar{X}}\cdot \bar{L}_{2}=E_{2}\cdot\bar{L}_{2}=1,%
$$
which implies that $E_{1}\cdot \bar{L}_{2}=E_{3}\cdot \bar{L}_{2}=E_{4}\cdot
\bar{L}_{2}=E_{5}\cdot \bar{L}_{2}=E_{6}\cdot
\bar{L}_{2}=\bar{C}\cdot\bar{L}_{2}=0$.

Let $\beta\colon\bar{X}\to\check{X}$ be a~contraction of
the~curves $\bar{L}_{2}$, $\bar{C}$, $E_6$, $E_5$ and $E_4$. Then
$$
\beta\big(E_{2}\big)\cdot\beta\big(E_{2}\big)=\beta\big(E_{3}\big)\cdot\beta\big(E_{3}\big)=-1,
$$
and $\check{X}$ is a~smooth del Pezzo surface such that
$K_{\check{X}}^{2}=6$, which implies that there are irreducible
smooth rational curves $\check{L}_{3}$ and
$\check{L}_{2}^{\prime}$ on the~surface $\check{X}$~such~that
$$
\check{L}_{3}\cdot\beta\big(E_{3}\big)=\check{L}_{2}^{\prime}\cdot\beta\big(E_{2}\big)=1%
$$
and
$\check{L}_{3}\cdot\check{L}_{3}=\check{L}_{2}^{\prime}\cdot\check{L}_{2}^{\prime}=-1$.
Let $\bar{L}_{3}$ and $\bar{L}_{2}^{\prime}$ be the~proper
transforms of the~curves $\check{L}_{3}$ and
$\check{L}_{2}^{\prime}$~on the~surface $\bar{X}$, respectively.
Then $\bar{L}_{3}\cdot \bar{L}_{3}=\bar{L}_{2}^{\prime}\cdot
\bar{L}_{2}^{\prime}=-1$ and
$$
-K_{\bar{X}}\cdot \bar{L}_{3}=-K_{\bar{X}}\cdot\bar{L}_{2}^{\prime}=E_{3}\cdot \bar{L}_{3}=E_{2}\cdot\bar{L}_{2}^{\prime}=1,%
$$
which implies that
$\bar{C}\cdot\bar{L}_{3}=\bar{C}\cdot\bar{L}_{2}^{\prime}=0$, and
$E_{i}\cdot \bar{L}_{3}=E_{j}\cdot \bar{L}_{2}^{\prime}=0$ for
every $i\ne 3$ and $j\ne 2$,

Put $\bar{L}_{4}=\tau(\bar{L}_{3})$,
$\bar{L}_{5}=\tau(\bar{L}_{2})$,
$\bar{L}_{5}^{\prime}=\tau(\bar{L}_{2}^{\prime})$. Then
$\bar{C}\cdot\bar{L}_{4}=\bar{C}\cdot\bar{L}_{5}=\bar{C}\cdot\bar{L}_{5}^{\prime}=0$
and
$$
-K_{\bar{X}}\cdot \bar{L}_{4}=-K_{\bar{X}}\cdot \bar{L}_{5}=-K_{\bar{X}}\cdot\bar{L}_{5}^{\prime}=E_{4}\cdot \bar{L}_{4}=E_{5}\cdot\bar{L}_{5}=E_{5}\cdot\bar{L}_{5}^{\prime}=1,%
$$
which implies that
$E_{i}\cdot\bar{L}_{5}=E_{i}\cdot\bar{L}_{5}^{\prime}=E_{j}\cdot
\bar{L}_{4}=0$ for every $i\ne 5$ and $j\ne 4$.

Put $L_{3}=\pi(\bar{L}_{3})$, $L_{4}=\pi(\bar{L}_{4})$,
$L_{2}=\pi(\bar{L}_{2})$,
$L_{2}^{\prime}=\pi(\bar{L}_{2}^{\prime})$,
$L_{5}=\pi(\bar{L}_{5})$,
$L_{5}^{\prime}=\pi(\bar{L}_{5}^{\prime})$. Then
$$
L_{3}+L_{4}\sim L_{2}+L_{5}\sim L_{2}^{\prime}+L_{5}^{\prime}\sim -2K_{X},%
$$
and $\mathrm{c}(X,L_{3}+L_{4})=1/3$, which implies that
$\mathrm{lct}_{2}(X)\leqslant 2/3$.

Note that
$\mathrm{c}(X,L_{2}+L_{5})=\mathrm{c}(X,L_{2}^{\prime}+L_{5}^{\prime})=1/2$.

Suppose that $D\ne (L_{3}+L_{4})/2$. To complete the~proof, it is
enough to show that~$\mu>2/3$.

Suppose that $\mu\leqslant 2/3$. Let us derive a~contradiction.

It follows from $(\ref{equation:cyclic-orbifolds})$ that
$a_1\leqslant 6/7$, $a_2\leqslant 10/7$, $a_3\leqslant 12/7$,
$a_4\leqslant 12/7$, $a_5\leqslant 10/7$, $a_6\leqslant 6/7$.

By Remark~\ref{remark:convexity}, without loss of generality we
may assume that $\bar{L}_{4}\not\subset\mathrm{Supp}(D)$. Then
$$
1-a_4=\bar{L}_3\cdot\bar{D}\geqslant 0,
$$
which gives us $a_{4}\leqslant 1$. Similarly, we may assume that
either $\bar{L}_{2}\not\subset\mathrm{Supp}(D)$ or
$\bar{L}_{5}\nobreak\not\subset\nobreak\mathrm{Supp}(D)$, which
implies that either $a_2\leqslant 1$ or $a_{5}\leqslant 1$,
respectively.

Let us show that $L_2+L_2^{\prime}+L_3\sim -3K_{X}$. We can easily
see that
$$
\bar{L}_2\sim_{\mathbb{Q}}\pi^*\big(L_2\big)-\frac{5}{7}E_1-\frac{10}{7}E_2-\frac{8}{7}E_3-\frac{6}{7}E_4-\frac{4}{7}E_5-\frac{2}{7}E_6,\\
$$
$$
\bar{L}_2^{\prime}\sim_{\mathbb{Q}}\pi^*\big(L_2^{\prime}\big)-\frac{5}{7} E_1-\frac{10}{7} E_2-\frac{8}{7} E_3-\frac{6}{7}E_4-\frac{4}{7}E_5-\frac{2}{7}E_6,\\
$$
$$
\bar{L}_3\sim_{\mathbb{Q}}\pi^*\big(L_3\big)-\frac{4}{7}E_1-\frac{8}{7} E_2-\frac{12}{7} E_3-\frac{9}{7}E_4-\frac{6}{7}E_5-\frac{3}{7}E_6,\\
$$
which implies that $L_2+L_2^{\prime}+L_3\sim_{\mathbb{Q}}
-3K_{X}$, since $\mathrm{Pic}(X)\cong\mathbb{Z}^{3}$ and
$$
L_{2}\cdot L_{2}=\frac{3}{7},\ L_2^{\prime}\cdot L_2^{\prime}=\frac{3}{7},\ L_3\cdot L_3=\frac{5}{7},\ L_2^{\prime}\cdot L_3=\frac{8}{7},\ L_2\cdot L_3=\frac{8}{7},\ L_2\cdot L_2^{\prime}=\frac{10}{7},%
$$
but $L_2+L_2^{\prime}+L_3$ is a~Cartier divisor, which implies
that $L_2+L_2^{\prime}+L_3\sim -3K_{X}$.

Since $\mathrm{c}(X,L_2+L_2^{\prime}+L_3)=1/4$, we may assume that
$\mathrm{Supp}(D)$ does not contain at least one curve among
$L_{2}$, $L_{2}^{\prime}$ and $L_{3}$ by
Remark~\ref{remark:convexity}, which implies that either
$a_{2}\leqslant 1$ or $a_{3}\leqslant 1$.

It follows from $(\ref{equation:cyclic-orbifolds})$ and
$a_4\leqslant 2$ that $\mu a_{i}<1$ for every $i$. By
Lemma~\ref{lemma:cyclic-orbifolds-blow-up}, there exists a~point
$$
Q\in\Big\{E_{2}\cap E_{3}, E_{3}\cap E_{4}, E_{4}\cap E_{5}\Big\},
$$
such that $(\ref{equation:log-pull-back})$ is not Kawamata log
terminal at the~point $Q\in\bar{X}$, but it is Kawamata
log~terminal elsewhere. Take $k\in\{2,3,4\}$ such that
$Q=E_{k}\cap E_{k+1}$. It follows from
Lemma~\ref{lemma:adjunction}~that
$$
\left\{%
\aligned
&2a_k-a_{k-1}-a_{k+1}=\bar{D}\cdot E_k\geqslant\mathrm{mult}_Q\Big(\bar{D}\cdot E_k\Big)>\frac{1}{\mu}-a_{k+1}>\frac{3}{2}-a_{k+1},\\%
&2a_{k+1}-a_{k}-a_{k+2}=\bar{D}\cdot E_{k+1}\geqslant\mathrm{mult}_Q\Big(\bar{D}\cdot E_{k+1}\Big)>\frac{1}{\mu}-a_{k}\geqslant\frac{3}{2}-a_{k},\\%
\endaligned\right.%
$$
which is impossible by $(\ref{equation:cyclic-orbifolds})$, since
$a_{4}\leqslant 1$, and either $a_{2}\leqslant 1$ or
$a_{3}\leqslant 1$.
\end{proof}

\begin{lemma}
\label{lemma:single-point-A7-reducible} Suppose that $m=7$. Then
the~following conditions are equivalent:
\begin{itemize}
\item the~curve $R$ is irreducible,%
\item the~surface $\bar{X}$ contains an~irreducible curve $\bar{L}_{4}$ such that $\bar{L}_{4}\cdot \bar{L}_{4}=-1$ and $\bar{L}_{4}\cdot E_{4}=1$.%
\item the~surface $\bar{X}$ contains an~irreducible  curve
$\bar{L}_{4}$ such that $\bar{L}_{4}\cdot \bar{L}_{4}=-1$,
$\bar{L}_{4}\cdot E_{4}=1$~and
$$
\omega\circ\pi\big(\bar{L}_{4}\big)\subset\mathrm{Supp}\big(R\big).
$$
\end{itemize}
\end{lemma}

\begin{proof}
Suppose that $\bar{X}$ has  an~irreducible curve $\bar{L}_{4}$
such that $\bar{L}_{4}\cdot \bar{L}_{4}=-1$ and $\bar{L}_{4}\cdot
E_{4}=1$.~Then
$$
\bar{L}_4\sim_{\mathbb{Q}}\pi^*\big(L_4\big)-\frac{1}{2}E_1-E_2-\frac{3}{2}E_3-2E_4-\frac{3}{2}E_5-E_6-\frac{1}{2} E_7,%
$$
where $L_{4}=\pi(\bar{L}_{4})$. Then
$\tau(\bar{L}_{4})=\bar{L}_{4}$ and
$\omega(L_{4})\subset\mathrm{Supp}(R)$, because
$$
-1+\bar{L}_{4}\cdot\tau\big(\bar{L}_{4}\big)=\bar{L}_{4}\cdot\Big(\bar{L}_4+\tau\big(\bar{L}_{4}\big)\Big)=\bar{L}_{4}\cdot\Big(\pi^*\big(-2K_{X}\big)-E_1-2E_2-3E_3-4E_4-3E_5-2E_6-E_7\Big)=-2.%
$$

Suppose now that the~curve $R$ is reducible. Let us show that
the~surface $\bar{X}$ contains an~irreducible curve $\bar{L}_{4}$
such that $\bar{L}_{4}\cdot \bar{L}_{4}=-1$ and $\bar{L}_{4}\cdot
E_{4}=1$.

Let $\eta\colon\bar{X}\to\bar{X}^{\prime}$ be a~contraction of
the~curve $\bar{C}$. Then there is a~commutative~diagram
$$
\xymatrix{
&\bar{X}\ar@{->}[rr]^{\pi}\ar@{->}[dl]_{\eta}&& X\ar@{->}[rr]^{\omega}&&\mathbb{P}(1,1,2)\ar@{^{(}->}[rr]^{\phi}&&\mathbb{P}^{3}\ar@{-->}[ddll]^{\psi}\\
\bar{X}^{\prime}\ar@{->}[drrr]_{\pi^{\prime}}&&&&&\\
&&&X^{\prime}\ar@{->}[rr]^{\omega^{\prime}} &&\mathbb{P}^{2}}%
$$
where $\pi^{\prime}$ is a~minimal resolution, $\phi$ is
an~anticanonical embedding, $\psi$ is a~projection from
$\phi\circ\omega(P)$, and $\omega^{\prime}$ is a~double cover
branched at $\psi\circ\phi(R)$. Note that $X^{\prime}$ is a~del
Pezzo surface and $K_{X^{\prime}}^2=2$.

The morphism $\pi^{\prime}$ contracts the~smooth curves
$\eta(E_{2})$, $\eta(E_{3})$, $\eta(E_{4})$, $\eta(E_{5})$ and
$\eta(E_{6})$. But
$$
\eta\big(E_{2}\big)\in\mathrm{Sing}\big(X^{\prime}\big),
$$
and $X^{\prime}$ has a~singularity of type $\mathbb{A}_{5}$ at
the~point $\eta(E_{2})$. Put $P^{\prime}=\eta(E_{2})$.

Put $R^{\prime}=\psi\circ\phi(R)$. Then $R^{\prime}$ is reducible,
since $R$ is reducible.

Since $\mathrm{Sing}(\mathbb{P}(1,1,2))\not\in R$, one of the
following cases hold:
\begin{itemize}
\item either $\phi(R)$ is a~union of a smooth conic and an~irreducible quartic,%
\item or the~curve $\phi(R)$ is a~union of three different smooth conics.%
\end{itemize}

The case when the~curve $\phi(R)$ consists of a~union of three
different smooth conics is impossible, since the~surface
$X^{\prime}$ has a~singularity of type $\mathbb{A}_{5}$ at
the~point $P^{\prime}=\mathrm{Sing}(X^{\prime})$.

We see that the~curve $\phi(R)$ is a~union of a~smooth conic and
an~irreducible quartic curve, which easily implies that
$R^{\prime}$ is a~union of a~line $L$ and an~irreducible cubic
curve $Z$. Then
$$
\mathrm{mult}_{\omega^{\prime}(P^{\prime})}\Big(L\cdot Z\Big)=3,
$$
because $X^{\prime}$ has a~singularity of type $\mathbb{A}_{5}$ at
the~point $P^{\prime}$. Then $\bar{X}$ contains a curve
$\bar{L}_{4}$ such that
$$
\omega^{\prime}\circ\pi^{\prime}\circ\eta\big(\bar{L}_{4}\big)=L,
$$
and $\bar{L}_{4}$ is irreducible. Then $\bar{L}_{4}\cdot
\bar{L}_{4}=-1$ and $\bar{L}_{4}\cdot E_{4}=1$.
\end{proof}

The proof of Lemma~\ref{lemma:single-point-A7-reducible} can be
simplified using the~results obtained in
\cite[Section~2]{Park-Won}.

\begin{lemma}
\label{lemma:single-point-A7-irrreducible} Suppose that $m=7$ and
$R$ is irreducible. Then $\mu\geqslant\mathrm{lct}_{3}(X)=3/5$.
\end{lemma}

\begin{proof}
Arguing as in   the~proofs of Lemmas~\ref{lemma:single-point-A5}
and \ref{lemma:single-point-A6}, we see that there is
an~irreducible smooth rational curve $\bar{L}_{2}$ on the~surface
$\bar{X}$ such that $\bar{L}_{2}\cdot \bar{L}_{2}=-1$ and
$$
-K_{\bar{X}}\cdot \bar{L}_{2}=E_{2}\cdot\bar{L}_{2}=1,%
$$
which implies that $E_{1}\cdot \bar{L}_{2}=E_{3}\cdot
\bar{L}_{2}=E_{4}\cdot \bar{L}_{2}=E_{5}\cdot
\bar{L}_{2}=E_{6}\cdot \bar{L}_{2}=E_{7}\cdot
\bar{L}_{2}=\bar{C}\cdot\bar{L}_{2}=0$.

Put $\bar{L}_{5}=\tau(\bar{L}_{2})$. Then $\bar{L}_{5}\cdot
\bar{L}_{5}=-1$ and $-K_{\bar{X}}\cdot
\bar{L}_{5}=E_{5}\cdot\bar{L}_{5}=1$, which implies that
$$
E_{1}\cdot \bar{L}_{5}=E_{2}\cdot \bar{L}_{5}=E_{3}\cdot\bar{L}_{5}=E_{4}\cdot \bar{L}_{5}=E_{6}\cdot\bar{L}_{5}=E_{7}\cdot \bar{L}_{5}=\bar{C}\cdot\bar{L}_{5}=0.%
$$

Since the~branch curve $R$ is reducible by
Lemma~\ref{lemma:single-point-A7-reducible}, one can show that there exists
an~irreducible smooth rational curve $\bar{L}_{3}$ on the~surface $\bar{X}$
such that $\bar{L}_{3}\cdot \bar{L}_{3}=-1$ and
$$
-K_{\bar{X}}\cdot \bar{L}_{3}=E_{3}\cdot\bar{L}_{3}=1,%
$$
which implies that $E_{1}\cdot \bar{L}_{3}=E_{2}\cdot
\bar{L}_{3}=E_{4}\cdot \bar{L}_{3}=E_{5}\cdot
\bar{L}_{3}=E_{6}\cdot \bar{L}_{3}=E_{7}\cdot
\bar{L}_{3}=\bar{C}\cdot\bar{L}_{3}=0$.

Put $\bar{L}_{6}=\tau(\bar{L}_{2})$,
$\bar{L}_{5}=\tau(\bar{L}_{3})$, $L_{2}=\pi(\bar{L}_{2})$,
$L_{3}=\pi(\bar{L}_{4})$, $L_{5}=\pi(\bar{L}_{5})$ and
$L_{6}=\pi(\bar{L}_{6})$. Then
$$
\bar{L}_2\sim_{\mathbb{Q}}\pi^*\big(L_2\big)-\frac{3}{4}E_1-\frac{3}{2} E_2-\frac{5}{4}E_3-E_4-\frac{3}{4}E_5-\frac{1}{2}E_6-\frac{1}{4}E_7,%
$$
$$
\bar{L}_3\sim_{\mathbb{Q}}\pi^*\big(L_3\big)-\frac{5}{8}E_1-\frac{5}{4}E_2-\frac{15}{8}E_3-\frac{3}{2}E_4-\frac{9}{8}E_5-\frac{3}{4}E_6-\frac{3}{8}E_7,%
$$
$$
\bar{L}_5\sim_{\mathbb{Q}}\pi^*\big(L_5\big)-\frac{3}{8}E_1-\frac{3}{4}E_2-\frac{9}{8}E_3-\frac{3}{2}E_4-\frac{15}{8}E_5-\frac{5}{4}E_6-\frac{5}{8} E_7,%
$$
$$
\bar{L}_6\sim_{\mathbb{Q}}\pi^*\big(L_6\big)-\frac{1}{4}E_1-\frac{1}{2}E_2-\frac{3}{4}E_3-E_4-\frac{5}{4}E_5-\frac{3}{2}E_6-\frac{3}{4} E_7,%
$$
which implies that $L_2+2L_3\sim -3K_{X}$. Indeed, we have $L_2 +
2L_3\sim_{\mathbb{Q}}-3K_{X}$, since
$$
L_2\cdot L_{2}=\frac{1}{2},\ L_3\cdot L_{3}=\frac{7}{8},\ L_2\cdot L_{3}=\frac{5}{4},%
$$
and $\mathrm{Pic}(X)\cong\mathbb{Z}^{3}$. But $L_2+2L_3$ is
a~Cartier divisor, which implies that $L_2+2L_3\sim -3K_{X}$.

We have $\mathrm{c}(X,L_2+2L_3)=3/15$ and $L_2 + 2L_3\sim
-3K_{X}$, which implies that $\mathrm{lct}_{3}(X)\leqslant 3/5$.

To complete the~proof, it is enough to show that~$\mu\geqslant
3/5$.

Suppose that $\mu<3/5$. Let us derive a~contradiction.

By Remark~\ref{remark:convexity}, we may assume that the~support
of the~divisor $\bar{D}$ does not contain at least one components
of every curve $\bar{L}_{2}+\bar{L}_{6}$,
$\bar{L}_{2}+2\bar{L}_{3}$, $\bar{L}_{3}+\bar{L}_{5}$. But
$$
\bar{D}\cdot\bar{L}_{i}=1-a_{i},
$$
which implies that $a_{i}\leqslant 1$ if
$\bar{L}_{i}\not\subset\mathrm{Supp}(\bar{D})$. Therefore, either
$a_{3}\leqslant 1$ or $a_{2}\leqslant 1$ and $a_{5}\leqslant 1$.

If $a_{3}\leqslant 1$, then it follows from
$(\ref{equation:cyclic-orbifolds})$~that
$$
a_1\leqslant\frac{7}{8},\ a_2\leqslant \frac{6}{5},\ a_3\leqslant 1,\ a_4\leqslant\frac{4}{3},\ a_5\leqslant\frac{5}{3},\ a_6\leqslant\frac{3}{2},\ a_7\leqslant\frac{7}{8}.%
$$

If $a_{2}\leqslant 1$ and $a_{5}\leqslant 1$, then it follows from
$(\ref{equation:cyclic-orbifolds})$~that
$$
a_1\leqslant\frac{7}{8},\ a_2\leqslant 1,\ a_3\leqslant \frac{3}{2},\ a_4\leqslant\frac{4}{3},\ a_5\leqslant 1,\ a_6\leqslant\frac{6}{5},\ a_7\leqslant\frac{7}{8}.%
$$

By Lemma~\ref{lemma:cyclic-orbifolds-blow-up}, there exists
$k\in\{2,3,4,5\}$ such that $(\ref{equation:log-pull-back})$ is
not Kawamata~log~terminal at the~point $E_{k}\cap E_{k+1}$ and is
Kawamata log terminal outside of $E_{k}\cap E_{k+1}$.

Put $Q=E_{k}\cap E_{k+1}$. Then it follows from
Lemma~\ref{lemma:adjunction} that
$$
\left\{%
\aligned
&2a_k-a_{k-1}-a_{k+1}=\bar{D}\cdot E_k\geqslant\mathrm{mult}_Q\Big(\bar{D}\cdot E_k\Big)>\frac{1}{\mu}-a_{k+1}>\frac{5}{3}-a_{k+1},\\%
&2a_{k+1}-a_{k}-a_{k+2}=\bar{D}\cdot E_{k+1}\geqslant\mathrm{mult}_Q\Big(\bar{D}\cdot E_{k+1}\Big)>\frac{1}{\mu}-a_{k}\geqslant\frac{5}{3}-a_{k},\\%
\endaligned\right.%
$$
which is impossible by $(\ref{equation:cyclic-orbifolds})$, since
we assume that either $a_{3}\leqslant 1$ or $a_{2}\leqslant 1$ and
$a_{5}\leqslant 1$.
\end{proof}

\begin{lemma}
\label{lemma:single-point-A7-reducible-2} Suppose that $m=7$ and
$R$ is reducible. Then $\mu\geqslant\mathrm{lct}_{2}(X)=1/2$.
\end{lemma}

\begin{proof}
By Lemma~\ref{lemma:single-point-A7-reducible}, the surface $X$
contains an~irreducible curve $\bar{L}_{4}$ such that
$$
\omega\circ\pi\big(\bar{L}_{4}\big)\subset\mathrm{Supp}\big(R\big)
$$
and $-\bar{L}_{4}\cdot \bar{L}_{4}=\bar{L}_{4}\cdot E_{4}=1$. Then
$-K_{\bar{X}}\cdot\bar{L}_{4}=1$, which implies that
$$
E_{1}\cdot\bar{L}_{4}=E_{2}\cdot \bar{L}_{4}=E_{3}\cdot\bar{L}_{4}=E_{5}\cdot \bar{L}_{4}=E_{6}\cdot\bar{L}_{4}=E_{7}\cdot \bar{L}_{4}=\bar{C}\cdot\bar{L}_{4}=0.%
$$

Put $L_{4}=\pi(\bar{L}_{4})$. Then $2L_4\sim -2K_{X}$ and
$$
\bar{L}_4\sim_{\mathbb{Q}}\pi^*\big(L_4\big)-\frac{1}{2}E_1-E_2-\frac{3}{2}E_3-2E_4-\frac{3}{2}E_5-E_6-\frac{1}{2} E_7,%
$$
which implies that
$\mathrm{lct}_{2}(X)\leqslant\mathrm{c}(X,L_{4})=1/2$.

To complete the~proof, it is enough to show that~$\mu\geqslant
1/2$.

Suppose that $\mu<1/2$. Let us derive a~contradiction.

By Remark~\ref{remark:convexity}, we may assume that
$L_{4}\not\subset\mathrm{Supp}(D)$. Then
$$
0\leqslant\bar{L}_{4}\cdot\bar{D}=1-a_{4},
$$
which implies that  $a_{4}\leqslant 1$. Thus, it follows from
$(\ref{equation:cyclic-orbifolds})$  that
$$
a_1\leqslant\frac{7}{8},\ a_2\leqslant\frac{3}{2},\ a_3\leqslant \frac{5}{4},\ a_4\leqslant 1,\ a_5\leqslant\frac{5}{4},\ a_6\leqslant\frac{3}{2},\ a_7\leqslant\frac{7}{8}.%
$$

It follows from Lemma~\ref{lemma:cyclic-orbifolds-blow-up} that
there exists a~point
$$
Q\in\Big\{E_{2}\cap E_{3}, E_{3}\cap E_{4}, E_{4}\cap E_{5}, E_{5}\cap E_{6}\Big\}%
$$
such that $\mathrm{LCS}(\bar{X},\mu\bar{D}+\sum_{i=1}^{7}\mu
a_iE_i)=Q$.

Without loss of generality, we may assume that either $Q=E_{2}\cap
E_{3}$ or $Q=E_{3}\cap E_{4}$.

If $Q=E_3\cap E_4$, then it follows from
Lemma~\ref{lemma:adjunction} that
$$
2a_4-a_3-a_5=\bar{D}\cdot E_4\geqslant\mathrm{mult}_{Q}\Big(\bar{D}\cdot E_4 \Big)>\frac{1}{\mu}-a_3>2-a_3,%
$$
which together with $(\ref{equation:cyclic-orbifolds})$ imply that
$a_4>1$, which is a~contradiction.

If $Q=E_2\cap E_3$, then it follows from
Lemma~\ref{lemma:adjunction} that
$$
2a_3-a_2-a_4=\bar{D}\cdot E_3\geqslant\mathrm{mult}_Q\Big(\bar{D}\cdot E_3\Big)>\frac{1}{\mu}-a_2>2-a_2,%
$$
which together with $(\ref{equation:cyclic-orbifolds})$
immediately leads to a~contradiction.
\end{proof}

\begin{lemma}
\label{lemma:single-point-A8} Suppose that $m=8$. Then
$\mu\geqslant\mathrm{lct}_{3}(X)=1/2$.
\end{lemma}

\begin{proof}
Arguing as in   the~proofs of Lemmas~\ref{lemma:single-point-A5}
and \ref{lemma:single-point-A6}, we see that there is
an~irreducible smooth rational curve $\bar{L}_{3}$ on the~surface
$\bar{X}$ such that $\bar{L}_{3}\cdot \bar{L}_{3}=-1$ and
$$
-K_{\bar{X}}\cdot \bar{L}_{3}=E_{3}\cdot\bar{L}_{3}=1,%
$$
which implies that $E_{1}\cdot \bar{L}_{3}=E_{2}\cdot
\bar{L}_{3}=E_{4}\cdot \bar{L}_{3}=E_{5}\cdot
\bar{L}_{3}=E_{6}\cdot \bar{L}_{3}=E_{7}\cdot
\bar{L}_{3}=\bar{C}\cdot\bar{L}_{3}=0$.

Put $\bar{L}_{6}=\tau(\bar{L}_{3})$. Then $\bar{L}_{6}\cdot
\bar{L}_{6}=-1$ and $-K_{\bar{X}}\cdot
\bar{L}_{6}=E_{6}\cdot\bar{L}_{6}=1$, which implies that
$$
E_{1}\cdot \bar{L}_{6}=E_{2}\cdot \bar{L}_{6}=E_{3}\cdot\bar{L}_{6}=E_{4}\cdot \bar{L}_{6}=E_{5}\cdot\bar{L}_{6}=E_{7}\cdot \bar{L}_{6}=\bar{C}\cdot\bar{L}_{6}=0.%
$$

Put $L_{3}=\pi(\bar{L}_{3})$ and $L_{6}=\pi(\bar{L}_{6})$. Then
$3L_{3}\sim 3L_{6}\sim -3K_{X}$. On the~other hand, we have
$$
\bar{L}_3\sim_{\mathbb{Q}}\pi^*\big(L_3\big)-\frac{2}{3}E_1-\frac{4}{3}E_2-2E_3-\frac{5}{3}E_4-\frac{4}{3}E_5-E_6-\frac{2}{3}E_7-\frac{1}{3}E_8,%
$$
$$
\bar{L}_6\sim_{\mathbb{Q}}\pi^*\big(L_6\big)-\frac{1}{3}E_1-\frac{2}{3} E_2-E_3-\frac{4}{3}E_4-\frac{5}{3} E_5-2 E_6-\frac{4}{3} E_7-\frac{2}{3}E_8,%
$$
which implies $\mathrm{c}(X,L_{3})=\mathrm{c}(X,L_{6})=1/2$. Then
$\mathrm{lct}_{3}(X)\leqslant 1/2$.

To complete the~proof, it is enough to show that~$\mu\geqslant
1/2$.

Suppose that $\mu<1/2$. Let us derive a~contradiction.

By Remark~\ref{remark:convexity}, we may assume that
$\mathrm{Supp}(\bar{D})$ does not contain $\bar{L}_{3}$ and
$\bar{L}_{6}$. Then
$$
1-a_{3}=\bar{D}\cdot\bar{L}_{3}\geqslant 0,
$$
which implies that $a_{3}\leqslant 1$. Similarly, we have
$a_{6}\leqslant 1$. Then it follows from
$(\ref{equation:cyclic-orbifolds})$~that
$$
a_1\leqslant\frac{8}{9},\ a_2 \leqslant\frac{7}{6},\ a_3\leqslant 1,\ a_4\leqslant\frac{4}{3}, a_5\leqslant\frac{4}{3},\ a_6\leqslant 1,\ a_7\leqslant\frac{7}{6},\ a_8\leqslant \frac{8}{9}.%
$$

By Lemma~\ref{lemma:cyclic-orbifolds-blow-up}, there exists
$k\in\{2,3,4,5,6\}$ such that $(\ref{equation:log-pull-back})$ is
not Kawamata log terminal at the~point $E_{k}\cap E_{k+1}$ and is
Kawamata log terminal outside of the~point  $E_{k}\cap E_{k+1}$.

Put $Q=E_{k}\cap E_{k+1}$. Then it follows from
Lemma~\ref{lemma:adjunction} that
$$
\left\{%
\aligned
&2a_k-a_{k-1}-a_{k+1}=\bar{D}\cdot E_k\geqslant\mathrm{mult}_Q\Big(\bar{D}\cdot E_k\Big)>\frac{1}{\mu}-a_{k+1}>\frac{1}{2}-a_{k+1},\\%
&2a_{k+1}-a_{k}-a_{k+2}=\bar{D}\cdot E_{k+1}\geqslant\mathrm{mult}_Q\Big(\bar{D}\cdot E_{k+1}\Big)>\frac{1}{\mu}-a_{k}\geqslant\frac{1}{2}-a_{k},\\%
\endaligned\right.%
$$
which is impossible by $(\ref{equation:cyclic-orbifolds})$, since
$a_{3}\leqslant 1$ and $a_{6}\leqslant 1$.
\end{proof}

The assertion of Theorem~\ref{theorem:main-single-point} is
proved.

\section{One non-cyclic singular point}
\label{section:non-cyclic-orbifolds}

Let $X$ be a~sextic surface in $\mathbb{P}(1,1,2,3)$ with
canonical singularities such that $|\mathrm{Sing}(X)|=1$, and
$\mathrm{Sing}(X)$ consists of a~singular point of type
$\mathbb{D}_{4}$, $\mathbb{D}_{5}$, $\mathbb{D}_{6}$,
$\mathbb{D}_{7}$, $\mathbb{D}_{8}$, $\mathbb{E}_{6}$,
$\mathbb{E}_{7}$ or $\mathbb{E}_{8}$.

\begin{theorem}
\label{theorem:main-single-point-non-cyclic} The following equality holds:
$$
\mathrm{lct}\big(X\big)=\left\{\aligned
&\mathrm{lct}_{2}\big(X\big)=1/3\ \text{if $P$ is a~point of type $\mathbb{D}_{8}$},\\
&\mathrm{lct}_{2}\big(X\big)=2/5\ \text{if $P$ is a~point of type $\mathbb{D}_{7}$},\\
&\mathrm{lct}_{1}\big(X\big)\ \text{in the~remaining cases}.\\
\endaligned
\right.
$$
\end{theorem}

\begin{corollary}
\label{corollary:main-single-point-non-cyclic} The inequality
$\mathrm{lct}(X)\leqslant 1/2$ holds.
\end{corollary}

In the~rest of this section we will prove
Theorem~\ref{theorem:main-single-point-non-cyclic}.

Let $D$ be an~effective $\mathbb{Q}$-divisor on $X$~such~that
$D\sim_{\mathbb{Q}} -K_{X}$. We must show that
$$
\mathrm{c}\big(X,D\big)\geqslant\left\{\aligned
&\mathrm{lct}_{2}\big(X\big)=1/3\ \text{if $P$ is a~ point of type $\mathbb{D}_{8}$},\\
&\mathrm{lct}_{2}\big(X\big)=2/5\ \text{if $P$ is a~ point of type $\mathbb{D}_{7}$},\\
&\mathrm{lct}_{1}\big(X\big)\ \text{in the~remaining cases}.\\
\endaligned
\right.
$$
to prove Theorem~\ref{theorem:main-single-point-non-cyclic}. Put
$\mu=\mathrm{c}(X,D)$.

Suppose that $\mu<\mathrm{lct}_{1}(X)$.~Then $\mathrm{LCS}(X,\mu
D)=\mathrm{Sing}(X)$ by Lemma~\ref{lemma:smooth-points}. Put
$P=\mathrm{Sing}(X)$.

Let $\pi\colon\bar{X}\to X$ be a~minimal resolution, let
$E_{1},E_{2}\ldots,E_{m}$ be irreducible $\pi$-exceptional curves,
let $C$ be the~curve in $|-K_{X}|$ such that $P\in C$, and let
$\bar{C}$ be its~proper transform~on~$\bar{X}$.~Then
$$
\bar{C}\sim_{\mathbb{Q}}\pi^*\big(C\big)-\sum_{i=1}^{m}n_{i}E_i,
$$
where $n_{i}\in\mathbb{N}$. Without loss of generality, we may assume that
$E_{3}\cdot\sum_{i\ne 3}E_{i}=3$. Then
$$
\mathrm{lct}_{1}\big(X\big)=\mathrm{c}\big(X,C\big)=\frac{1}{n_{3}}=\left\{\aligned
&1/2\ \text{if $P$ is of type  $\mathbb{D}_{4}$, $\mathbb{D}_{5}$, $\mathbb{D}_{6}$, $\mathbb{D}_{7}$ or $\mathbb{D}_{8}$},\\
&1/3\ \text{if $P$ is of type  $\mathbb{E}_{6}$},\\
&1/4\ \text{if $P$ is of type  $\mathbb{E}_{7}$},\\
&1/6\ \text{if $P$ is of type  $\mathbb{E}_{8}$}.\\
\endaligned
\right.
$$

By Remark~\ref{remark:convexity}, we may assume that
$C\not\subset\mathrm{Supp}(D)$, since the~curve $C$ is irreducible.

Let $\bar{D}$ be the~proper transform of the~divisor $D$ on
the~surface $\bar{X}$. Then
$$
\bar{D}\sim_{\mathbb{Q}}\pi^*\big(D\big)-\sum_{i=1}^{m}a_iE_i,
$$
where $a_{i}$ is a~non-negative rational number. Then
$$
K_{\bar{X}}+\mu\Big(\bar{D}+\sum_{i=1}^{m}a_iE_i\Big)\sim_{\mathbb{Q}}\pi^*\Big(K_X+\mu D\Big),%
$$
which implies that $(\bar{X},\mu\bar{D}+\sum_{i=1}^{m}\mu a_iE_i)$ is not
Kawamata log terminal (see Remark~\ref{remark:log-pull-back}).

\begin{lemma}
\label{lemma:non-cyclic-blow-up} The equality $\mu a_3=1$ holds.
\end{lemma}

\begin{proof}
The equality $\mu a_3=1$ follows from Lemma~\ref{lemma:Prokhorov}.
\end{proof}

\begin{lemma}
\label{lemma:single-point-D4} Suppose that $P$ is not a~point of
type $\mathbb{E}_{6}$, $\mathbb{E}_{7}$ or $\mathbb{E}_{8}$. Then
$$
\mu\geqslant\left\{\aligned
&\mathrm{lct}_{2}\big(X\big)=1/3\ \text{if $P$ is a~ point of type $\mathbb{D}_{8}$},\\
&\mathrm{lct}_{2}\big(X\big)=2/5\ \text{if $P$ is a~ point of type $\mathbb{D}_{7}$},\\
\endaligned
\right.
$$
and $P$ is either a~point of type $\mathbb{D}_{7}$ or is a~point
of type $\mathbb{D}_{8}$.
\end{lemma}

\begin{proof}
Without loss of generality, we may assume that the~diagram
$$
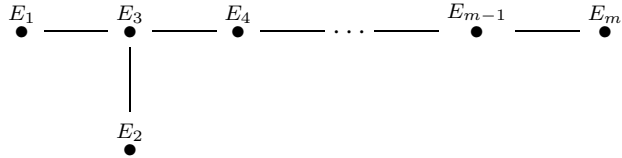
\xymatrix{
\stackrel{E_1}{\bullet} \ar@{-}[r] & \stackrel{E_3}{\bullet}\ar@{-}[r] \ar@{-}[d] & \stackrel{E_4}{\bullet}\ar@{-}[r] & \cdots &\ar@{-}[l] \stackrel{E_{m-1}}{\bullet} \ar@{-}[r] & \stackrel{E_m}{\bullet}\\
& \stackrel{E_2}{\bullet} &}
$$
shows how the~$\pi$-exceptional curves intersect each other. Then
$$
\bar{C}\sim_{\mathbb{Q}}\pi^*\big(C\big)-E_{1}-E_{2}-E_{m}-\sum_{i=3}^{m-1}2E_i,
$$
which implies that $\bar{C}\cdot E_{m-1}=1$ and $\bar{C}\cdot E_{i}=0\iff i\ne
m-1$. Then
\begin{equation}
\label{equation:D-m-orbifolds} \left\{\aligned
&1-a_{m-1}=\bar{D}\cdot\bar{C}\geqslant 0,\\
&2a_1-a_3=\bar{D}\cdot E_{1}\geqslant 0,\\
&2a_2-a_3=\bar{D}\cdot E_{2}\geqslant 0,\\
&2a_3-a_1-a_2-a_3=\bar{D}\cdot E_{3}\geqslant 0,\\
&\cdots\\
&2a_{m-1}-a_{m-2}-a_m=\bar{D}\cdot E_{m-1}\geqslant 0,\\
&2a_m-a_{m-1}=\bar{D}\cdot E_{m}\geqslant 0,\\
\endaligned
\right.
\end{equation}
which easily implies that $a_{3}\leqslant 2$ if $m\leqslant 6$. But $\mu
a_{3}=1$ and $\mu<\mathrm{lct}_{1}(X)=1/2$ by
Lemma~\ref{lemma:non-cyclic-blow-up}, which implies that either $m=7$ or $m=8$.

Arguing as in   the~proofs of Lemmas~\ref{lemma:single-point-A5}
and \ref{lemma:single-point-A6}, we may assume that there is
an~irreducible smooth rational curve $\bar{L}_{1}$ on the~surface
$\bar{X}$ such that $\bar{L}_{1}\cdot \bar{L}_{1}=-1$~and
$$
-K_{\bar{X}}\cdot \bar{L}_{1}=E_{1}\cdot\bar{L}_{1}=1,%
$$
which implies that $\bar{C}\cdot\bar{L}_{1}=0$ and $E_{i}\cdot
\bar{L}_{1}=0\iff i\ne 1$.

Let $\omega\colon X\to\mathbb{P}(1,1,2)$ be the~natural double
cover given by $|-2K_{X}|$, and let $\tau$ be
a~biregular~involution of the~surface $\bar{X}$ that is induced by
$\omega$. Put $\bar{L}_{2}=\tau(\bar{L}_{1})$.~If $m=7$, then
$$
-K_{\bar{X}}\cdot\bar{L}_{2}=E_{2}\cdot\bar{L}_{2}=1%
$$
and $\bar{L}_{2}\cdot \bar{L}_{2}=-1$, which implies that
$\bar{C}\cdot\bar{L}_{2}=0$ and $E_{i}\cdot \bar{L}_{2}=0\iff i\ne
2$.

Put $L_{1}=\pi(\bar{L}_{1})$ and $L_{2}=\pi(\bar{L}_{2})$. Then
$L_{1}+L_{2}\sim -2K_{X}$. If $m=7$, then
$$
\bar{L}_1\sim_{\mathbb{Q}}\pi^*\big(L_1\big)-\frac{7}{4}E_1-\frac{5}{4}E_2-\frac{5}{2}E_3-2E_4-\frac{3}{2}E_5-E_6-\frac{1}{2}E_7,%
$$
$$
\bar{L}_2\sim_{\mathbb{Q}}\pi^*\big(L_2\big)-\frac{5}{4}E_1-\frac{7}{4}E_2-\frac{5}{2}E_3-2E_4-\frac{3}{2}E_5-E_6-\frac{1}{2}E_7,%
$$
which implies that $\mathrm{c}(X,L_{1}+L_{2})=1/5$ and
$\mathrm{lct}_{2}(X)\leqslant 2/5$. If $m=7$, then
$$
a_{3}\leqslant\frac{5}{2}
$$
by $(\ref{equation:D-m-orbifolds})$. But  $\mu a_{3}=1$ by
Lemma~\ref{lemma:non-cyclic-blow-up}. Then $\mu\geqslant 2/5$ if $m=7$, which
is exactly what we need.

We may assume that $m=8$. Then  $\bar{L}_{2}=\bar{L}_{1}$ and
$$
\bar{L}_1\sim_{\mathbb{Q}}\pi^*\big(L_1\big)-2E_1-\frac{3}{2}E_2-3E_3-\frac{5}{2}E_4-2E_5-\frac{3}{2}E_6-E_7-\frac{1}{2}E_8,%
$$
which implies that
$\mathrm{lct}_{2}(X)\leqslant\mathrm{c}(X,L_{1})=1/3$. But
$a_{3}\leqslant 1/3$ by $(\ref{equation:D-m-orbifolds})$ and $\mu
a_{3}=1$ by Lemma~\ref{lemma:non-cyclic-blow-up}, which implies
that $\mu\geqslant 1/3$, which completes the~proof since
$\mathrm{lct}_{2}(X)\geqslant\mathrm{lct}(X)$.
\end{proof}

To complete the~proof of
Theorem~\ref{theorem:main-single-point-non-cyclic}, we may assume
that $P$ is a~point of type $\mathbb{E}_{6}$,
$\mathbb{E}_{7}$~or~$\mathbb{E}_{8}$.

Without loss of generality, we may assume that the~diagram
$$
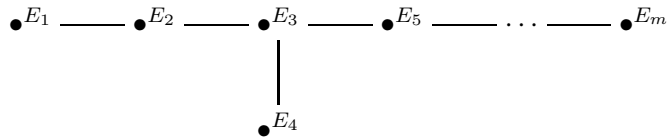
\xymatrix{ {\bullet}^{E_1} \ar@{-}[r] & {\bullet}^{E_2} \ar@{-}[r]  & {\bullet}^{E_3} \ar@{-}[r] \ar@{-}[d] & {\bullet}^{E_5} \ar@{-}[r]& \cdots &\ar@{-}[l] {\bullet}^{E_m}\\
 & & {\bullet}^{E_4} & }
$$
shows how the~$\pi$-exceptional curves intersect each other. It is well-known
(cf. \cite{Park2002}\cite{Park-Won-PEMS}) that
\begin{itemize}
\item if $m=6$, then $\bar{C}\cdot E_4=1$, which implies that and $\bar{C}\cdot E_{i}=0\iff i\ne 4$,%
\item if $m=7$, then $\bar{C}\cdot E_1=1$, which implies that and $\bar{C}\cdot E_{i}=0\iff i\ne 1$,%
\item if $m=8$, then $\bar{C}\cdot E_8=1$, which implies that and $\bar{C}\cdot E_{i}=0\iff i\ne 8$.%
\end{itemize}

Put $k=4$ if $m=6$, put $k=1$ if $m=7$, put $k=8$ if $m=8$. Then
\begin{equation}
\label{equation:E6-E7-E8-orbifolds} \left\{\aligned
&1-a_{k}=\bar{D}\cdot\bar{C}\geqslant 0,\\
&2a_1-a_3=\bar{D}\cdot E_{1}\geqslant 0,\\
&2a_2-a_3-a_1=\bar{D}\cdot E_{2}\geqslant 0,\\
&2a_3-a_2-a_4-a_5=\bar{D}\cdot E_{3}\geqslant 0,\\
&2a_4-a_3=\bar{D}\cdot E_{4}\geqslant 0,\\
&2a_5-a_3-a_6=\bar{D}\cdot E_{5}\geqslant 0,\\
&\cdots\\
&2a_{m-1}-a_{m-2}-a_m=\bar{D}\cdot E_{m-1}\geqslant 0,\\
&2a_m-a_{m-1}=\bar{D}\cdot E_{m}\geqslant 0,\\
\endaligned
\right.
\end{equation}
which implies that $a_{3}<n_{3}$. But $n_3=1/\mathrm{lct}_{1}(X)$ and $\mu
a_{3}=1$ by Lemma~\ref{lemma:non-cyclic-blow-up}. Then $\mu\geqslant
\mathrm{lct}_{1}(X)$.

The assertion of Theorem~\ref{theorem:main-single-point-non-cyclic} is proved.

\section{Many singular points}
\label{section:orbifolds-with-many-singular-points}

Let $X$ be a~sextic surface in $\mathbb{P}(1,1,2,3)$ with
canonical singularities such that $|\mathrm{Sing}(X)|\geqslant 2$.

\begin{theorem}
\label{theorem:main-many-points} The following equality holds:
$$
\mathrm{lct}\big(X\big)=\left\{\aligned
&\mathrm{lct}_{2}\big(X\big)=1/2\ \text{if $\mathrm{Sing}(X)$ consists of a~point of type $\mathbb{A}_{7}$ and a~point of type $\mathbb{A}_{1}$},\\
&\mathrm{lct}_{2}\big(X\big)=2/3\ \text{if $X$ has a~singular point of type $\mathbb{A}_{6}$},\\
&\mathrm{lct}_{2}\big(X\big)=2/3\ \text{if $X$ has a~singular point of type $\mathbb{A}_{5}$},\\
&\mathrm{lct}_{2}\big(X\big)=\mathrm{min}\big(\mathrm{lct}_{1}\big(X\big),4/5\big)\ \text{if $X$ has a~singular point of type $\mathbb{A}_{4}$},\\
&\mathrm{lct}_{1}\big(X\big)\ \text{in the~remaining cases},\\
\endaligned
\right.
$$
and if there exists an~effective $\mathbb{Q}$-divisor $D$ on
the~surface $X$ such that $D\sim_{\mathbb{Q}}-K_{X}$ and
$$
\mathrm{c}\big(X,D\big)=\mathrm{lct}\big(X\big)=\frac{2}{3},
$$
then either $D$ is an~irreducible curve in $|-K_{X}|$ with a~cusp
at a~point in $\mathrm{Sing}(X)$ of type~$\mathbb{A}_{2}$, or
the~divisor $D$ is uniquely defined  and it can be explicitly
described.
\end{theorem}

Let $D$ be an~arbitrary effective $\mathbb{Q}$-divisor on
the~surface $X$~such~that
$$
D\sim_{\mathbb{Q}} -K_{X},
$$
and put $\mu=\mathrm{c}(X,D)$. To prove
Theorem~\ref{theorem:main-many-points}, it is enough to show that
$$
\mu\geqslant\left\{\aligned
&\mathrm{lct}_{2}\big(X\big)=1/2\ \text{if $\mathrm{Sing}(X)$ consists of a~point of type $\mathbb{A}_{7}$ and a~point of type $\mathbb{A}_{1}$},\\
&\mathrm{lct}_{2}\big(X\big)=2/3\ \text{if $X$ has a~singular point of type $\mathbb{A}_{6}$},\\
&\mathrm{lct}_{2}\big(X\big)=2/3\ \text{if $X$ has a~singular point of type $\mathbb{A}_{5}$},\\
&\mathrm{lct}_{2}\big(X\big)=\mathrm{min}\big(\mathrm{lct}_{1}\big(X\big),4/5\big)\ \text{if $X$ has a~singular point of type $\mathbb{A}_{4}$},\\
&\mathrm{lct}_{1}\big(X\big)\ \text{in the~remaining cases},\\
\endaligned
\right.
$$
and if $\mu=\mathrm{lct}(X)=2/3$, then we have the~following two
possibilities:
\begin{itemize}
\item either $D$ is a~curve in $|-K_{X}|$ with a~cusp at a~point in $\mathrm{Sing}(X)$ of type~$\mathbb{A}_{2}$,%
\item or the~divisor $D$ is uniquely defined and it can be explicitly described.%
\end{itemize}

\begin{lemma}
\label{lemma:many-points-bad-points} If $\mathrm{Sing}(X)$ has
a~point of type $\mathbb{D}_4$, $\mathbb{D}_5$, $\mathbb{D}_6$,
$\mathbb{E}_6$, $\mathbb{E}_7$ or $\mathbb{E}_8$, then
$\mu\geqslant\mathrm{lct}_{1}(X)$.
\end{lemma}

\begin{proof}
Suppose that $\mathrm{Sing}(X)$ has a~point of type
$\mathbb{D}_4$, $\mathbb{D}_5$, $\mathbb{D}_6$, $\mathbb{E}_6$,
$\mathbb{E}_7$ or $\mathbb{E}_8$, but
$\mu<\mathrm{lct}_{1}(X)$.~Then
$$
\mathrm{LCS}\big(X,\mu D\big)\subsetneq\mathrm{Sing}\big(X\big)
$$
and $\mathrm{LCS}(X,\mu D)$ consists of a~point in
$\mathrm{Sing}(X)$ that is not of type $\mathbb{A}_{1}$ or
$\mathbb{A}_{2}$ by Lemma~\ref{lemma:smooth-points}.

If the~locus $\mathrm{LCS}(X,\mu D)$ is a~singular point of
the~surface $X$ of type $\mathbb{D}_4$, $\mathbb{D}_5$,
$\mathbb{D}_6$, $\mathbb{E}_6$, $\mathbb{E}_7$ or $\mathbb{E}_8$,
then arguing as in the~proof of
Theorem~\ref{theorem:main-single-point-non-cyclic}, we immediately
obtain a~contradiction.

By Remark~\ref{remark:DP1-SING}, the~locus $\mathrm{LCS}(X,\mu D)$
must be a~singular point of the~surface $X$ of type
$\mathbb{A}_{3}$, and we can easily obtain a~contradiction arguing
as in the~proof of Corollary~\ref{corollary:A3-point}.
\end{proof}

\begin{lemma}
\label{lemma:many-points-good-points} Suppose that
$\mathrm{Sing}(X)$ consists of points of type $\mathbb{A}_1$,
$\mathbb{A}_2$ or $\mathbb{A}_3$. Then
$\mu\geqslant\mathrm{lct}_{1}(X)$.~If
$$
\mu=\mathrm{lct}_{1}\big(X\big)=\frac{2}{3},
$$
then $D$ is a~curve in $|-K_{X}|$ with a~cusp at a~point in
$\mathrm{Sing}(X)$ of type~$\mathbb{A}_{2}$.
\end{lemma}

\begin{proof}
This follows from Lemma~\ref{lemma:smooth-points} and the~proof of
Corollary~\ref{corollary:A3-point}.
\end{proof}

By Remark~\ref{remark:DP1-SING} and Lemmas~\ref{lemma:many-points-bad-points}
and \ref{lemma:many-points-bad-points}, we may assume that
$$
\mathrm{Sing}\big(X\big)\in\left\{\aligned %
&\mathbb{A}_7+\mathbb{A}_1, \mathbb{A}_6+\mathbb{A}_1, \mathbb{A}_5+\mathbb{A}_1, \mathbb{A}_5+\mathbb{A}_1+\mathbb{A}_1, \mathbb{A}_5+\mathbb{A}_2, \mathbb{A}_5+\mathbb{A}_2+\mathbb{A}_1,\\%
&\mathbb{A}_4+\mathbb{A}_4, \mathbb{A}_4+\mathbb{A}_3, \mathbb{A}_4+\mathbb{A}_2+\mathbb{A}_1, \mathbb{A}_4+\mathbb{A}_2, \mathbb{A}_4+\mathbb{A}_1+\mathbb{A}_1, \mathbb{A}_4+\mathbb{A}_1,\\%
\endaligned\right\},
$$
which implies that there is a~point $P\in\mathrm{Sing}(X)$ that is
a~point of type $\mathbb{A}_m$ for $m\in\{4,5,6,7\}$.

Let $\pi\colon\bar{X}\to X$ be a~minimal resolution, let
$E_{1},E_{2},\ldots,E_{m}$ be $\pi$-exceptional curves such~that
$$
E_{i}\cdot E_{j}\ne 0\iff \big|i-j\big|\leqslant 1
$$
and $\pi(E_{i})=P$ for every $i\in\{1,\ldots,m\}$, let $C$ be
the~unique curve in $|-K_{X}|$ such that $P\in C$, and let
$\bar{C}$ be the~proper transform of the~curve $C$ on the~surface
$\bar{X}$. Then
$$
\bar{C}\cdot E_{1}=\bar{C}\cdot E_{m}=1,
$$
and $\bar{C}\cdot E_{2}=\bar{C}\cdot E_{3}=\cdots=\bar{C}\cdot E_{m-1}=0$. Note
that $\bar{C}\cong\mathbb{P}^{1}$ and $\bar{C}\cdot\bar{C}=-1$.

Let $\bar{D}$ be the~proper transform of $D$ on the~surface $\bar{X}$. Then
$$
\bar{D}\sim_{\mathbb{Q}}\pi^*\big(D\big)-\sum_{i=1}^{m}a_iE_i,
$$
where $a_{i}$ is a~non-negative rational number. Then
\begin{equation}
\label{equation:many-points-orbifolds} \left\{\aligned
&1-a_1-a_m=\bar{D}\cdot\bar{C}\geqslant 0,\\
&2a_1 - a_2=\bar{D}\cdot E_{1}\geqslant 0,\\
&\cdots\\
&2a_{m-1}-a_{m-2}-a_m=\bar{D}\cdot E_{m-1}\geqslant 0,\\
&2a_m - a_{m-1}=\bar{D}\cdot E_{m}\geqslant 0,\\
\endaligned
\right.
\end{equation}

Let $\eta\colon\bar{X}\to\bar{X}^{\prime}$ be a~contraction of
the~curve $\bar{C}$. Then there is a~commutative~diagram
$$
\xymatrix{
&\bar{X}\ar@{->}[rr]^{\pi}\ar@{->}[dl]_{\eta}&& X\ar@{->}[rr]^{\omega}&&\mathbb{P}(1,1,2)\ar@{^{(}->}[rr]^{\phi}&&\mathbb{P}^{3}\ar@{-->}[ddll]^{\psi}\\
\bar{X}^{\prime}\ar@{->}[drrr]_{\pi^{\prime}}&&&&&\\
&&&X^{\prime}\ar@{->}[rr]^{\omega^{\prime}} &&\mathbb{P}^{2}}%
$$
where $\omega$ and $\omega^{\prime}$ are natural double covers
$\pi^{\prime}$ is a~minimal resolution, $\phi$ is an~anticanonical
embedding, and $\psi$ is a~projection from $\phi\circ\omega(P)$.
Put $P^{\prime}=\eta(E_{2})$. Then
$P^{\prime}\in\mathrm{Sing}(X^{\prime})$.

\begin{remark}
\label{remark:diagram-morphism-etc} The birational morphism
$\pi^{\prime}$ contracts the~smooth curves
$\eta(E_{2}),\eta(E_{3}),\ldots,\eta(E_{m-1})$,
and~$\pi^{\prime}\circ\eta$ contracts all $\pi$-exceptional curves
that are different from the~curves $E_{1},E_{2},\ldots,E_{m}$.
\end{remark}

Let $R$ be the branch curve in $\mathbb{P}(1,1,2)$ of the double
cover $\omega$. Put $R^{\prime}=\psi\circ\phi(R)$.

\begin{lemma}
\label{lemma:many-points-A7-A1} Suppose that $m=7$. Then
$\mu\geqslant\mathrm{lct}_{2}(X)=1/2$.
\end{lemma}

\begin{proof}
Let $\alpha\colon\bar{X}\to\breve{X}$ be a~contraction of
the~irreducible curves $\bar{C}$, $E_7$, $E_6$, $E_5$, $E_4$,
$E_3$ and $E_2$, and let $F$ be the~$\pi$-exceptional curve such
that $\pi(F)$ is a~point of type $\mathbb{A}_{1}$. Then
$$
\breve{X}\cong\mathbb{P}\Big(\mathcal{O}_{\mathbb{P}^{1}}\oplus\mathcal{O}_{\mathbb{P}^{1}}\big(2\big)\Big).
$$

Let $\breve{L}_{2}$ be the~fiber of the~projection
$\breve{X}\to\mathbb{P}^{1}$ such that
$\alpha(\bar{C})\in\breve{L}_{2}$,
and~let~$\bar{L}_{2}$~be~the~proper transform~of~the~curve
$\breve{L}_{2}$~on~the~surface $\bar{X}$ via $\alpha$. Then
$\bar{L}_{2}\cdot \bar{L}_{2}=-1$ and
$$
-K_{\bar{X}}\cdot \bar{L}_{2}=E_{2}\cdot\bar{L}_{2}=F\cdot\bar{L}_{2}=1,%
$$
which implies that $E_{1}\cdot \bar{L}_{2}=E_{3}\cdot
\bar{L}_{2}=E_{4}\cdot \bar{L}_{2}=E_{5}\cdot
\bar{L}_{2}=E_{6}\cdot \bar{L}_{2}=E_{7}\cdot
\bar{L}_{2}=\bar{C}\cdot\bar{L}_{2}=0$.

Let $\beta\colon\bar{X}\to\check{X}$ be a~contraction of
the~curves $\bar{L}_{2}$, $E_{2}$, $\bar{C}$, $E_7$, $E_6$, $E_5$,
$E_4$. Then
$$
\beta\big(E_{3}\big)\cdot\beta\big(E_{3}\big)=\beta\big(F\big)\cdot\beta\big(F\big)=0,
$$
and $\check{X}$ is a~smooth del Pezzo surface such that
$K_{\check{X}}^{2}=8$. Then $\check{X}\cong\mathbb{P}^{1}\times
\mathbb{P}^{1}$.

Let $\check{L}_{4}$ be the~curve in $|\beta(F)|$ such that
$\beta(E_{4})\in \check{L}_{4}$, and let $\bar{L}_{3}$ be its
proper transform on the~surface $\bar{X}$ via $\beta$. Then one
can easily check that $\bar{L}_{4}\cdot \bar{L}_{4}=-1$ and
$$
-K_{\bar{X}}\cdot \bar{L}_{4}=E_{4}\cdot \bar{L}_{4}=1,%
$$
which implies that $E_{1}\cdot \bar{L}_{4}=E_{2}\cdot
\bar{L}_{4}=E_{3}\cdot \bar{L}_{4}=E_{5}\cdot
\bar{L}_{4}=E_{6}\cdot \bar{L}_{4}=E_{7}\cdot
\bar{L}_{4}=\bar{C}\cdot\bar{L}_{4}=F\cdot\bar{L}_{4}=0$.

Put $L_{4}=\pi(\bar{L}_{4})$. Then one can easily check that
$$
\bar{L}_4\sim_{\mathbb{Q}}\pi^*\big(L_4\big)-\frac{1}{2}E_1-E_2-\frac{3}{2}E_3-2E_4-\frac{3}{2}E_5-E_6-\frac{1}{2} E_7,%
$$
which implies that $\mathrm{c}(X,L_{4})=1/2$. But $2L_4\sim
-2K_{X}$, which implies that $\mathrm{lct}_{2}(X)\leqslant 1/2$.

Arguing as in the proof of
Lemma~\ref{lemma:single-point-A7-reducible}, we see that
$\omega(L_{4})\subset\mathrm{Supp}(R)$.

Arguing as in the~proof of
Lemma~\ref{lemma:single-point-A7-reducible-2} and using
$(\ref{equation:many-points-orbifolds})$, we see that
$\mu\geqslant\mathrm{lct}_{2}(X)=1/2$.
\end{proof}

\begin{lemma}
\label{lemma:many-points-A6-A1} Suppose that $m=6$. Then
$\mu\geqslant\mathrm{lct}_{2}(X)=2/3$, and if $\mu=2/3$, then
\begin{itemize}
\item either $D$ is a~curve in $|-K_{X}|$ with a~cusp at a~point in $\mathrm{Sing}(X)$ of type~$\mathbb{A}_{2}$,%
\item or the~divisor $D$ is uniquely defined and can be explicitly described.%
\end{itemize}
\end{lemma}

\begin{proof}
Let $\alpha\colon\bar{X}\to\breve{X}$ be a~contraction of
the~curves $\bar{C}$, $E_{6}$, $E_5$, $E_4$, $E_3$,
$E_{2}$.~Then~$\breve{X}$~is~a~smooth surface such that
$K_{\breve{X}}^{2}=7$, and $-K_{X}$ is nef. There is a~birational
morphism $\gamma\colon\breve{X}\to\hat{X}$~such~ that
$$
\hat{X}\cong\mathbb{P}\Big(\mathcal{O}_{\mathbb{P}^{1}}\oplus\mathcal{O}_{\mathbb{P}^{1}}\big(2\big)\Big),
$$
and $\gamma$ is a~blow down of a~smooth irreducible rational curve
that does not contain the~point~$\alpha(\bar{C})$.

Let $\hat{L}_{2}$ be the~fiber of the~projection
$\hat{X}\to\mathbb{P}^{1}$ such that
$\gamma\circ\alpha(\bar{C})\in\hat{L}_{2}$,
and~let~$\bar{L}_{2}$~be~the~proper transform~of~the~curve
$\hat{L}_{2}$~on~the~surface $\bar{X}$ via $\gamma\circ\alpha$.
Then $\bar{L}_{2}\cdot \bar{L}_{2}=-1$ and
$$
-K_{\bar{X}}\cdot \bar{L}_{2}=E_{2}\cdot\bar{L}_{2}=1,%
$$
which implies that $E_{1}\cdot \bar{L}_{2}=E_{3}\cdot
\bar{L}_{2}=E_{4}\cdot \bar{L}_{2}=E_{5}\cdot
\bar{L}_{2}=E_{6}\cdot \bar{L}_{2}=\bar{C}\cdot\bar{L}_{2}=0$.

Let $\beta\colon\bar{X}\to\check{X}$ be a~contraction of
the~curves $\bar{L}_{2}$, $\bar{C}$, $E_6$, $E_5$, $E_4$, and let
$F$ be the~$\pi$-exceptional curve such that $\pi(F)$ is a~point
of type $\mathbb{A}_{1}$. Then
$$
\beta\big(E_{2}\big)\cdot\beta\big(E_{2}\big)=\beta\big(E_{3}\big)\cdot\beta\big(E_{3}\big)=\beta\big(F\big)\cdot\beta\big(F\big)=-1,
$$
and $\check{X}$ is a~smooth del Pezzo surface such that
$K_{\check{X}}^{2}=6$. Thus, there exists an~irreducible smooth
rational curve $\check{L}_{3}$ on the~surface
$\check{X}$~such~that $\check{L}_{3}\cdot\check{L}_{3}=-1$,
$\check{L}_{3}\cdot\beta(E_{3})=1$ and
$\check{L}_{3}\cdot\beta(F)=0$.

Let $\bar{L}_{3}$ be the~proper transform of the~curve
$\check{L}_{3}$ on the~surface $\bar{X}$. Then $\bar{L}_{3}\cdot
\bar{L}_{3}=-1$ and
$$
-K_{\bar{X}}\cdot \bar{L}_{3}=E_{3}\cdot \bar{L}_{3}=1,%
$$
which implies that $E_{1}\cdot \bar{L}_{3}=E_{2}\cdot
\bar{L}_{3}=E_{4}\cdot \bar{L}_{3}=E_{5}\cdot
\bar{L}_{3}=E_{6}\cdot \bar{L}_{3}=\bar{C}\cdot\bar{L}_{3}=F\cdot
\bar{L}_{3}=0$.

Put $\bar{L}_{4}=\tau(\bar{L}_{3})$ and
$\bar{L}_{5}=\tau(\bar{L}_{2})$. Then
$\bar{C}\cdot\bar{L}_{4}=\bar{C}\cdot\bar{L}_{5}=0$ and
$$
-K_{\bar{X}}\cdot\bar{L}_{4}=-K_{\bar{X}}\cdot \bar{L}_{5}=E_{4}\cdot \bar{L}_{4}=E_{5}\cdot\bar{L}_{5}=1,%
$$
which implies that $E_{i}\cdot\bar{L}_{5}=E_{j}\cdot
\bar{L}_{4}=0$ for every $i\ne 5$ and $j\ne 4$.

Put $L_{3}=\pi(\bar{L}_{3})$, $L_{4}=\pi(\bar{L}_{4})$,
$L_{2}=\pi(\bar{L}_{2})$ and $L_{5}=\pi(\bar{L}_{5})$. Then
$$
L_{3}+L_{4}\sim L_{2}+L_{5}\sim -2K_{X},%
$$
which implies that $\mathrm{c}(X,L_{3}+L_{4})=1/3$ and
$\mathrm{c}(X,L_{2}+L_{5})=1/2$. Then
$\mathrm{lct}_{2}(X)\leqslant 2/3$. But
$$
\bar{L}_2\sim_{\mathbb{Q}}\pi^*\big(L_2\big)-\frac{5}{7}E_1-\frac{10}{7}E_2-\frac{8}{7}E_3-\frac{6}{7}E_4-\frac{4}{7}E_5-\frac{2}{7}E_6-\frac{1}{2}F,\\
$$
$$
\bar{L}_3\sim_{\mathbb{Q}}\pi^*\big(L_3\big)-\frac{4}{7}E_1-\frac{8}{7} E_2-\frac{12}{7} E_3-\frac{9}{7}E_4-\frac{6}{7}E_5-\frac{3}{7}E_6,\\
$$
which implies that $\mathrm{c}(X,2L_2+L_3)=1/4$. Then
$2L_2+L_3\sim_{\mathbb{Q}} -3K_{X}$, since
$\mathrm{Pic}(X)\cong\mathbb{Z}^{2}$ and
$$
L_{2}\cdot L_{2}=\frac{3}{7},\ L_3\cdot L_3=\frac{5}{7},\ L_2\cdot L_3=\frac{8}{7},%
$$
but $2L_2+L_3$ is a~Cartier divisor, which implies that
$2L_2+L_3\sim -3K_{X}$.

If $D$ is not a~curve in $|-K_{X}|$ and $D\ne (L_{3}+L_{4})/2$,
then arguing~as~in~the~proof~of~Lemma~\ref{lemma:single-point-A6},
we easily see that $\mu>2/3$, since we can use
$(\ref{equation:many-points-orbifolds})$. The lemma is proved (see
Example~\ref{example:Jihun}).
\end{proof}

\begin{lemma}
\label{lemma:many-points-A5} Suppose that $m=5$. Then
$\mu\geqslant\mathrm{lct}_{2}(X)=2/3$, and if $\mu=2/3$, then
\begin{itemize}
\item either $D$ is a~curve in $|-K_{X}|$ with a~cusp at a~point in $\mathrm{Sing}(X)$ of type~$\mathbb{A}_{2}$,%
\item or the~divisor $D$ is uniquely defined and can be explicitly described.%
\end{itemize}
\end{lemma}

\begin{proof}
The curve $R^{\prime}$ has an~ordinary tacnodal singularity at
the~point $\omega^{\prime}(P^{\prime})$, which implies that there
exists a~line $L^{\prime}\subset\mathbb{P}^{2}$~such~that either
$L^{\prime}\subset\mathrm{Supp}(R^{\prime})$ or
$L^{\prime}\not\subset\mathrm{Supp}(R^{\prime})$ and
$$
\mathrm{mult}_{\omega^{\prime}(P^{\prime})}\Big(L^{\prime}\cdot R^{\prime}\Big)=4.%
$$

There are irreducible smooth rational curves $L^{\prime}_{3}$ and
$L^{\prime}_{4}$ on the~surface $X^{\prime}$ such that
$$
\omega^{\prime}\big(L^{\prime}_{3}\big)=\omega^{\prime}\big(L^{\prime}_{4}\big)=L^{\prime}
$$
and $L^{\prime}_{3}=L^{\prime}_{4}\iff
L^{\prime}\subset\mathrm{Supp}(R^{\prime})$. Note that neither
$L^{\prime}_{3}$ nor $L^{\prime}_{4}$ contains a~point in
$\mathrm{Sing}(X^{\prime})\setminus R^{\prime}$.

Let $\bar{L}^{\prime}_{3}$ be the~proper transform of the~curve
$L^{\prime}_{3}$ on the~surface $\bar{X}^{\prime}$.~Then
$$
\bar{L}^{\prime}_{3}\cap\eta\big(E_{1}\big)=\bar{L}^{\prime}_{3}\cap\eta\big(E_{2}\big)=\bar{L}^{\prime}_{3}\cap\eta\big(E_{4}\big)=\bar{L}^{\prime}_{3}\cap\eta\big(E_{5}\big)=\varnothing,
$$
and $\bar{L}^{\prime}_{3}\cdot\eta(E_{3})=1$. Let
$\bar{L}^{\prime}_{4}$ be the~proper transform of the~curve
$L^{\prime}_{4}$ on the~surface $\bar{X}^{\prime}$.~Then
$$
\bar{L}^{\prime}_{4}\cap\eta\big(E_{1}\big)=\bar{L}^{\prime}_{4}\cap\eta\big(E_{2}\big)=\bar{L}^{\prime}_{4}\cap\eta\big(E_{4}\big)=\bar{L}^{\prime}_{4}\cap\eta\big(E_{5}\big)=\varnothing,
$$
and $\bar{L}^{\prime}_{4}\cdot\eta(E_{3})=1$. One can also check that
$\bar{L}^{\prime}_{3}\cap\bar{L}^{\prime}_{4}=\varnothing$ if
$\bar{L}^{\prime}_{3}\ne\bar{L}^{\prime}_{4}$.

Let $\bar{L}_{3}$ and $\bar{L}_{4}$ be the~proper transforms of
the~curves $\bar{L}^{\prime}_{3}$ and $\bar{L}^{\prime}_{4}$ on
the~surface $\bar{X}$, respectively, and let us put
$L_{3}=\pi(\bar{L}_{3})$ and $L_{4}=\pi(\bar{L}_{4})$. Then
$$
\bar{L}_{3}+\bar{L}_{4}\sim -2K_{X}
$$
and $\mathrm{c}(X, \bar{L}_{3}+\bar{L}_{4})=1/3$, which implies that
$\mathrm{lct}_{2}(X)\leqslant 2/3$.

If $D\ne (\bar{L}_{3}+\bar{L}_{4})/2$, then
$(\ref{equation:many-points-orbifolds})$, the~proof of
Lemma~\ref{lemma:single-point-A5} and
Lemma~\ref{lemma:smooth-points} imply that
$$
\mu\geqslant\mathrm{lct}_{2}\big(X\big)=\frac{2}{3}.
$$
and if $\mu=2/3$, then $D$ is a~curve in $|-K_{X}|$ with a~cusp at
a~point in $\mathrm{Sing}(X)$ of type~$\mathbb{A}_{2}$.
\end{proof}

\begin{lemma}
\label{lemma:many-points-A4} Suppose that $m=4$. Then
$$
\mu\geqslant\mathrm{lct}_{2}\big(X\big)=\mathrm{min}\big(\mathrm{lct}_{1}\big(X\big),4/5\big)\geqslant\frac{2}{3},
$$
and if $\mu=2/3$, then $D$ is a~curve in $|-K_{X}|$ with a~cusp at
a~point in $\mathrm{Sing}(X)$ of type~$\mathbb{A}_{2}$.
\end{lemma}

\begin{proof}
The point $\omega^{\prime}(P^{\prime})$ is an~ordinary cusp of
the~curve $R^{\prime}$. Then there is a~line
$L^{\prime}\subset\mathbb{P}^{2}$~such~that
$$
\mathrm{mult}_{\omega^{\prime}(P^{\prime})}\Big(L^{\prime}\cdot R^{\prime}\Big)=3.%
$$

Let $Z^{\prime}$ be a~curve in $X^{\prime}$ such that
$\omega^{\prime}(Z^{\prime})=L^{\prime}$ and $-K_{X^{\prime}}\cdot
Z^{\prime}=2$. Then
$$
Z^{\prime}\cap\mathrm{Sing}\big(X^{\prime}\big)=\mathrm{Sing}\big(Z^{\prime}\big)=R^{\prime},
$$
the $Z^{\prime}$ is irreducible curve that has an~ordinary cusp at
the~point $R^{\prime}$.

Let $\bar{Z}^{\prime}$ be the~proper transform of the~curve
$Z^{\prime}$ on the~surface $\bar{X}^{\prime}$. Then $Z^{\prime}$
is smooth and
$$
\eta\big(E_{2}\big)\cap\eta\big(E_{3}\big)\in\bar{Z}^{\prime}.
$$

Let $\bar{Z}$ be the~proper transform of the~curve
$\bar{Z}^{\prime}$ on the~surface $\bar{X}$. Put $Z=\pi(\bar{Z})$.
Then
$$
\bar{Z}\sim\pi^{*}\big(Z\big)-E_1-2E_2-2E_3-E_4
$$
and $E_2\cap E_3\in Z$. Then $\mathrm{c}(X, Z)=2/5$, which implies that
$\mathrm{lct}_{2}(X)\leqslant 4/5$.

Arguing as in the~proof of Lemma~\ref{lemma:single-point-A4} and
using Lemma~\ref{lemma:smooth-points} and
$(\ref{equation:many-points-orbifolds})$, we see that
$$
\mu\geqslant\mathrm{lct}_{2}\big(X\big)=\mathrm{min}\big(\mathrm{lct}_{1}\big(X\big),4/5\big)
$$
and if $\mu=2/3$, then $D$ is a~curve in $|-K_{X}|$ with a~cusp at
a~point in $\mathrm{Sing}(X)$ of type~$\mathbb{A}_{2}$.
\end{proof}

The~assertion of Theorem~\ref{theorem:main-many-points} is proved.

\end{document}